\documentclass[11pt]{amsart}

\usepackage[english]{babel}

\usepackage{amsmath}
\usepackage{graphicx}
\usepackage[colorlinks=true, allcolors=blue]{hyperref}
\usepackage[mathscr]{euscript}
\usepackage{amsthm}
\usepackage{amssymb}
\usepackage{amsfonts,amscd}
\usepackage[final]{showkeys}
\usepackage{fullpage}
\usepackage{caption}
\usepackage{subfig}
\usepackage{enumitem}
\usepackage{arydshln}
\setlist[enumerate,1]{label={\rm (\roman*)},leftmargin=2.5em}
\usepackage{multirow,multicol,booktabs}

\usepackage{rotating,pdflscape}

\usepackage{tikz}
\usetikzlibrary{cd,positioning,automata,arrows,shapes,matrix,calc,arrows, decorations.pathmorphing}
\tikzset{main node/.style={circle,draw,minimum size=0.3em,inner
sep=0.5pt}}
\tikzset{small node/.style={minimum size=0.1pt,inner
sep=0.1pt}}
\usepackage{xspace}
\usepackage{adjustbox}
\theoremstyle{plain}
\newtheorem{theorem}{Theorem}[section]
\newtheorem{lemma}[theorem]{Lemma}
\newtheorem{prop}[theorem]{Proposition}

\newtheorem{cor}[theorem]{Corollary}
\theoremstyle{definition}
\newtheorem{defn}[theorem]{Definition}
\newtheorem{exa}[theorem]{Example}

\theoremstyle{remark}
\newtheorem{rmk}[theorem]{Remark}
\numberwithin{equation}{section}

\newcommand{\al}{\alpha}
\newcommand{\be}{\beta}
\newcommand{\ep}{\varepsilon}

\newcommand{\R}{{\mathbb R}}
\newcommand{\Z}{{\mathbb Z}}

\newcommand{\ra}{\rightarrow}

\newcommand{\simproots}{\Pi}
\newcommand{\posroots}{{\Phi_+}}
\newcommand{\allroots}{\Phi}

\newcommand{\res}{\mathscr{D}}

\newcommand{\lc}{L^c}
\newcommand{\wnn}{w_{N}}
\newcommand{\mset}[1]{\{\!\{#1\}\!\}}
\newcommand{\seqd}{\underline{d}}

\newcommand{\calf}{\mathcal{F}}

\newcommand{\cali}{\mathcal{I}}
\newcommand{\calj}{\mathcal{J}}

\newcommand{\calm}{\mathcal{M}}
\newcommand{\cals}{\mathcal{S}}

\newcommand{\bfj}{\mathbf{j}}

\newcommand{\mr}{m}
\newcommand{\pr}{p}
\newcommand{\rr}{r}
\newcommand{\nine}{\mathbf{9}}
\newcommand{\ten}{\mathbf{10}}
\newcommand{\ob}{\Omega}
\newcommand{\fob}{\Omega_F}
\newcommand{\mob}{\Omega_M}
\newcommand{\nn}{\mathrm{NN}}
\newcommand{\pra}{P_R(\al)}

\DeclareMathOperator{\supp}{\rm{Supp}}
\DeclareMathOperator{\wt}{\rm{wt}}
\DeclareMathOperator{\interior}{\rm{Int}}

\newcommand{\se}{\subseteq}
\newcommand{\ip}[1]{\langle#1\rangle}

\newcommand{\inverse}{^{-1}}

\newcommand{\abs}[1]{\lvert #1 \rvert}

\makeatletter
\newcommand{\exterior}[1]{\mathop{\mathpalette\exterior@{#1}}}
\newcommand{\exterior@}[2]{%
  \raisebox{\depth}{%
  \fontsize{\sf@size}{0}%
  \m@th
  $\ifx#1\displaystyle\textstyle\else#1\fi\bigwedge$}%
  ^{\mspace{-2mu}#2}%
  \kern-\scriptspace
}
\makeatother

\title{Perfect matchings, Fano planes, and orthogonal bases of type $E_8$}
\author{R.M. Green and Tianyuan Xu}

\keywords{perfect matching, Fano plane, root system, generating function, Rothe
diagram}
\subjclass{Primary: 17B22; Secondary: 05A15, 05E18.}

\begin{document}
\begin{abstract}
We use perfect matchings and labelled Fano planes to construct and study the
$2025$ orthogonal bases of positive roots in the $E_8$ root system. The set of
these bases forms a highly structured, Bruhat-like graded poset $(\Omega, \leq_Q)$
whose rank function can be computed from the cardinalities of so-called
generalized Rothe diagrams. We give combinatorial characterizations of these
diagrams in terms of matchings and Fano planes, and we explain how to compute
the ranks of the elements of $\Omega$ using suitable combinatorial statistics such
as the weights of perfect matchings. As applications, we establish simple formulas for the rank
generating functions of $\Omega$ and of its 50 congruence classes under a natural
order congruence relation. In particular, we give a computer-free proof that the rank generating function
of $\ob$ factors as the product of the $q$-integers $[3]_q, [5]_q, [9]_q$, and
$[15]_q$.
%
\end{abstract}

\maketitle
\section{Introduction}\label{sec:intro}
The purpose of this paper is to use the combinatorics of perfect matchings and
Fano planes to gain insight into certain combinatorial structures associated
with the $E_8$ root system $\Phi=\Phi(E_8)$. Recall that the $E_8$ lattice gives
the densest possible packing of spheres in $8$-dimensional Euclidean space
\cite{viazovska17}, and that the $E_8$ root system consists of the $240$
elements of the lattice with minimal nonzero norm. The roots can be partitioned
into $120$ positive and $120$ negative roots by choosing a vector space total
order on the ambient vector space $V$. There are $2025$ orthogonal bases for $V$
that consist entirely of positive roots, and the set $\ob$ of all such bases
appears in the physics literature, where it has been used to give parity proofs
of the Kochen--Specker theorem in quantum mechanics \cite{waegell15}. 

The set $\ob$ is naturally a $W$-set for the Weyl group $W=W(E_8)$ associated to
$\Phi$ (Section \ref{sec:2025}). In previous work \cite{gx5}, we showed that
$\ob$ can be identified with a natural spanning set of a certain Macdonald
representation of $W$ \cite[Section 1]{gx5}. We defined a certain function
$\lambda: \ob\ra \Z$ that equips $\ob$ with the structure of a quasiparabolic
$W$-set, which means that $\ob$ forms a special type of $W$-set defined by
Rains--Vazirani \cite{rains13} that carries a graded, Bruhat-like partial order
whose rank function is given by $\lambda$. The function $\lambda$ is called the
level function in the context of quasiparabolic sets, and it can be used to
deform the $W$-action on $\ob$ in a suitable way to create a module for the
Iwahori--Hecke algebra associated to $W$ \cite[Theorem 7.1]{rains13}. In
\cite[Section 7.1]{gx5}, we showed using a computer that the generating function
$PS_X(q)=\sum_{R\in \ob} q^{\lambda(R)}$, known as the Poincar\'e polynomial of
$\ob$, has a nice factorization \begin{equation} \label{eq:ps_X}
PS_\ob(q)=[3]_q[5]_q[9]_q[15]_q \end{equation} into the quantum integers
$[d]_q=1+q+q^2+\dots+q^{d-1}$ for $d\in \{3,5,9,15\}$. In \cite{gx6}, we defined
a generalized Rothe diagram, $\res(R)$, for every element $R\in \ob$, and we
showed that $\lambda(R)$ equals the cardinality of $\res(R)$, thus giving
another interpretation of the level function $\lambda$.   

The results mentioned in the last paragraph have all been developed in the more
general setting of so-called $k$-roots \cite[Section 2.2]{gx6}. In particular,
these results have direct analogues in type $D_{2k}$ for any $k\in \Z_{\ge 2}$,
where the orthogonal bases of positive roots can be naturally constructed via
perfect matchings of the set $\{1,2, \dots, 2k\}$. The generalized Rothe
diagrams $\res(R)$ in type $D_{2k}$ recover the traditional Rothe diagrams of
permutations and the so-called involution Rothe diagrams associated to
fixed-point-free involutions in symmetric groups by Hamaker--Marberg--Pawlowski
\cite{hamaker20} (see \cite[Theorem 3.11 and Theorem 3.14]{gx6}), which is why
we name $\res(R)$ the generalized Rothe diagram in the more general setting. The
main contribution of this paper is an extension of this combinatorial framework
to type $E$.

To be more precise, we introduce two combinatorial models that make it possible
to understand the orthogonal bases in $\ob$ in the same concrete way that
perfect matchings allow us to understand the orthogonal bases in type $D_{2k}$.
These are the ``matching model'' and ``Fano model", and they provide a neat
dichotomy for the elements of $\ob$. We show that there exists a partition
$\ob=\ob_M \cup \ob_F$ into two parts, where $\ob_M$ can be naturally and
bijectively constructed from the 945 perfect matchings of the size-10 set
$\ten=\{0,1,2,\dots, 8,9\}$ (Definition \ref{defn:mob_construction}), and
$\ob_F$ from the 1080 labellings of the Fano plane using labels from the set
$\nine=\{1,2,\dots, 8,9\}$ (Definition \ref{defn:fob_construction}).
We use these models to give precise descriptions for the generalized Rothe
diagrams for the elements of $\ob$, in forms reminiscent of traditional Rothe
diagrams (Theorem \ref{thm:mob_res}, Theorem \ref{thm:fob_res}), and we use
these descriptions to obtain simple formulas for the level function (Proposition
\ref{prop:mob_level}, Proposition \ref{prop:fob_level}). In particular, we show
that if $R$ is an orthogonal basis in $\ob_M$ arising from a perfect matching
$M$, then the level of $R$ equals the weight of $M$ defined by
Deodhar--Srinivasan \cite{deodhar}, as well as the cardinality of the involution
Rothe diagram associated to $M$ by Hamaker--Marberg--Pawlowski \cite{hamaker20}. 

As applications of the matching and Fano models for $\ob$, we will give
conceptual proofs for a number of level generating functions. This includes the
proof of Equation \eqref{eq:ps_X}, which is in fact the original motivation of
this paper, but we will do much more: the map $\sigma : \ob \ra V$ defined by
$\sigma(R) = \Sigma_{\be \in R} \be$ gives a natural equivalence relation
$\sim_\sigma$ on $\ob$ whose equivalence classes are the fibres of $\sigma$, and
we will derive Equation \eqref{eq:ps_X} after first computing the level
generating function for each individual fibre. The relation $\sim_\sigma$ has
been studied in \cite[Section 5.5]{gx5} and enjoys many nice properties: it has
50 equivalence classes, each of which is an Eulerian interval with respect to
the quasiparabolic order $\le_Q$ on $\ob$, and $\sim_\sigma$ also gives
a poset congruence on $\ob$ with respect to $\le_Q$ in the sense of
\cite{reading02}. As we will explain in Section \ref{sec:sums}, the 50
equivalence classes form a distributive lattice that can be understood using
the natural partial order on the root lattice, or the left weak Bruhat order on
the Weyl group $W$ of $\Phi$, or the heap poset of a fully commutative element
in $W$, or a certain set $\cals$ of length-4 integer sequences with small
positive entries (Equation \eqref{eq:seqs}). In addition, the 50 equivalence
classes of $\sim_\sigma$ are highly compatible with the matching--Fano
dichotomy, with $\ob_M$ being the union of 42 of the equivalence classes and
$\ob_F$ the other 8. The dichotomy also propagates in natural ways to the
contexts of the aforementioned weak Bruhat order, heap poset, and integer
sequences, in ways that we will detail in Section \ref{sec:seq_to_sum}. The 42
equivalence classes that comprise $\mob$ are in bijection with the nonnesting
perfect matchings of the set $\ten$, and the details of the related Catalan
combinatorics also play an important role in the arguments of Section
\ref{sec:seq_to_sum}.

Our main results concerning generating functions are as follows. For each
sequence $\seqd\in \cals$, we will prove that the level generating function of the
corresponding equivalence class in $\ob$ can be expressed in a remarkably simple
formula in terms of $\seqd$ (Theorem \ref{thm:gen_func}). The matching and Fano
models for $\ob$ allow us to essentially reduce the proof of these formulas to
two known facts, one by Watson on a weight generating function of perfect
matchings \cite{watson14}, and another by the first-named author on a suitable
generating function on Fano planes labelled by any set of seven elements
\cite{the240}. We will use the generating functions for the
$\sim_\sigma$-classes to obtain the level generating functions for both $\ob_M$
and $\ob_F$, which we will then add up
to derive $PS_{\ob}(q)$ (Theorem \ref{thm:endgame}).

A key ingredient used for establishing the matching and Fano models is a
particular coordinatization of the root system $\Phi(E_8)$. This
coordinatization scheme,  due to Chevalley \cite[4.VII]{chevalley} (Section
\ref{sec:Phi}), is different from the more widely known Bourbaki conventions
used in standard references such as \cite{bourbaki} and \cite{humphreys90} and
used in our previous works \cite{gx5,gx6}. In hindsight, it is in fact fairly
straightforward to construct the elements of $\ob$ via perfect matchings and
labelled Fano planes using the Chevalley coordinates, although the calculation
of the generalized Rothe diagrams and level function still requires significant
work. Another nice feature of the Chevalley coordinates is that they lead to
very simple expressions, in terms of the matching and Fano models, for the sums
$\sigma(R)$ of the orthogonal bases $R\in \ob$ (Proposition
\ref{prop:seq_to_sum}).

The results of this paper have natural analogues for the root system of type
$E_7$ that can be proved using the same techniques (Remark \ref{rmk:e7}),  and
they also have applications for other $k$-roots by restriction (Remark
\ref{rmk:restriction}). Our proof for the level
generating function for the orthogonal bases given by the matching model
(Equation \eqref{eq:mob_gen}) readily generalizes to the context of orthogonal
bases of type $D_{2k}$. The definition of inversions of a labelled Fano plane (Definition \ref{def:fob_height}), which is new to our knowledge, may also be of independent interest as a counterpart to 
the well-studied notion of inversions of a permutation. 

%

The rest of the paper is organized as follows. We set up the background on the
set $\ob$ in Section \ref{sec:state}, including Chevalley's coordinatization of
the $E_8$ root system $\Phi$ in Section \ref{sec:Phi}, notions related to the
quasiparabolic structure of $\ob$ in Section \ref{sec:2025}, and useful facts
from \cite{gx5} on the equivalence classes of the relation $\sim_\sigma$.
Sections \ref{sec:mob} and \ref{sec:fob} are two parallel sections that develop
the details of the matching and Fano models for $\ob$, with each of these
sections consisting of three subsections that discuss how to use the  models to
construct orthogonal bases, determine the generalized Rothe diagrams of the
bases, and calculate the levels of the bases. Section \ref{sec:enum} deals with
level generating functions. Section \ref{sec:seq_to_sum} treats the technical
details we need in order to express the generating functions for the
$\sim_\sigma$-classes via the sequences of the set $\cals$, including how the
matching and Fano dichotomy of $\ob$ propagates to the sequences via suitable
bijections, and Section \ref{sec:gen_func} then combines these details with the
results of sections \ref{sec:mob}--\ref{sec:fob} to compute all the relevant
generating functions. Section \ref{sec:conclude} contains some concluding
remarks. Finally, the appendix displays two tables in landscape orientation that
are too wide to fit in the text in portrait orientation.

%
\section{Background}\label{sec:state}
This section introduces the main objects of study in this paper. In Section
\ref{sec:Phi}, we recall the details of the $E_8$ root system $\Phi$ and
introduce some related notation that will be used frequently in the paper. In
Section \ref{sec:2025}, we review results on the set $\ob$ of positive 8-roots
of $\Phi$. Section \ref{sec:sums} studies the set $\Sigma$ of all possible sums
of positive 8-roots, culminating in a chain of poset isomorphisms that reveals a
distributive lattice structure on $\Sigma$ and provides a simple model for
$\Sigma$ using certain integer sequences (Remark \ref{rmk:isos}).

\subsection{The $E_8$ root system}
\label{sec:Phi}
The primary object of study for this paper is the root system $\allroots$ of
type $E_8$. We refer the reader to \cite[Chapters 1 and 2]{humphreys90} for the
standard definitions related to root systems and their Weyl groups.

There are several well-known constructions of the root system $\Phi$, and we
will choose the following one, due to Chevalley, as it is the most convenient
for the purposes of the paper. Let $\nine = \{1, 2, \ldots, 9\}$ and let
$\{\ep_i : i \in \nine\}$ be an orthonormal basis for $\R^9$. Identify $\R^8$
with the orthogonal complement $V = \langle \bfj \rangle^\perp \leq \R^9$ of the
all-ones vector $\bfj := \sum_{i = 1}^9 \ep_i$, and define $e_i \in V$ to be the
orthogonal projection of $\ep_i$ to $V$ for each $i\in \nine$, so that we have
$\sum_{i = 1}^9 e_i = 0$. For any $i,j,k\in \nine$, let \[\pr_{ij} := e_i -
e_j,\quad \mr_{ij} := e_j - e_i, \quad \pr_{ijk} := e_i + e_j + e_k,
\quad\text{and }\quad \mr_{ijk} := -(e_i + e_j + e_k).\] Following Chevalley
\cite[4.VII]{chevalley}, we construct the $E_8$ root system as the set 
\[
    \allroots := \{\mr_{ij} :  i < j \} \cup \{\pr_{ij} :  i < j \} \cup
    \{\mr_{ijk} :  i < j < k \} \cup \{\pr_{ijk} : i < j < k \} \; (i,j,k\in
    \nine),
\] 
which is an irredundantly described set of $240$ vectors in $V$. The defining
Coxeter bilinear form $B$ on the vector space $V$ is given by the formula 
\[
   B(e_i, e_j) = \begin{cases} \frac{8}{9} & \text{if } i = j;\\
  -\frac{1}{9} & \text{otherwise}.
 \end{cases}
\]
   
Each element of $\Phi$ is a \emph{root}, and two roots $\al,\be$ are
\emph{orthogonal} if $B(\al,\be)=0$. Every root has a unique representation of
the form $e_i - e_j$ or $\pm(e_i + e_j + e_k)$, which we refer to as the {\it
standard form} of the root. We define the {\it support} of a root $\al$,
$\supp(\al)$, to be its support when written in standard form. In other words,
we have $\supp(\mr_{ij})=\supp(\pr_{ij})=\{i,j\}$ and
$\supp(\mr_{ijk})=\supp(\pr_{ijk})=\{i,j,k\}$ for all $i,j,k\in \nine$.

\begin{rmk}
\label{rmk:orthog}
As explained in \cite[Section 4.2]{vinberg88}, to compute $B(\al,\be)$ for two
roots $\al$ and $\be$, one may evaluate the bilinear product $B(\al, \be)$ as if
the vectors $e_i$ formed an orthonormal basis, which works in all cases unless both
$\al$ and $\be$ have the form $\pm(e_i + e_j + e_k)$, in which case we use the
formula 
\begin{equation}
\label{eq:overlap}
B(e_i + e_j + e_k, \ e_a + e_b + e_c) = |\{i, j, k\} \cap \{a, b, c\}| - 1 .
\end{equation}
In particular, two roots of the form $e_i - e_j$ are orthogonal if and only if
their supports are disjoint, two roots of the form $e_i + e_j + e_k$ are
orthogonal if and only if their supports overlap in a singleton, and the roots
$e_i - e_j$ and $e_a + e_b + e_c$ are orthogonal if and only if their supports
overlap in zero or two elements. \end{rmk}

The set 
\begin{equation}
\label{eq:pos}
\posroots = \{\pr_{ij} : 1 \leq i < j \leq 8\} \cup \{\mr_{i9} : 1
\leq i \leq 8\} \cup \{\pr_{ij9} : 1 \leq i < j \leq 8\} \cup \{\mr_{ijk} : 1
\leq i < j < k \leq 8\}
\end{equation}
forms a positive system in $\Phi$, and the
corresponding simple system, which forms a basis of $V$, is given by 
\begin{align*} 
  \simproots &= \{\al_1, \
  \al_3, \ \al_4, \ \al_5, \ \al_6, \ \al_7, \ \al_8, \ \al_2\}\\ &= \{\pr_{12},
  \ \pr_{23}, \ \pr_{34}, \ \pr_{45}, \ \pr_{56}, \ \pr_{67}, \ \pr_{78}, \
  \mr_{123} \} \\ &= \{e_1 - e_2, \ e_2 - e_3, \ e_3 - e_4, \ e_4 - e_5, \ e_5 -
e_6, \ e_6 - e_7, \ e_7 - e_8, \ -(e_1 + e_2 + e_3) \},
\end{align*} 
where the eight simple roots are given in the same order on the three lines. For
example, we have $\al_1=p_{12}=e_1-e_2$ and $\al_3=p_{23}=e_2-e_3$. We will fix
the choice of the positive and simple systems throughout this paper and refer to
the elements of $\posroots$ and $\simproots$ as the \emph{positive} and
\emph{simple} roots of $\Phi$, respectively. This choice of positive and simple
roots can be obtained from Chevalley's (and vice versa) by applying the
involutive orthogonal linear transformation $\omega : V \ra V$ satisfying
$\omega(e_i) = -e_{9-i}$ for all $1 \leq i \leq 8$. With respect to $\Pi$, the
Dynkin diagram of $\Phi$ is as shown in Figure \ref{fig:dynkin}. 

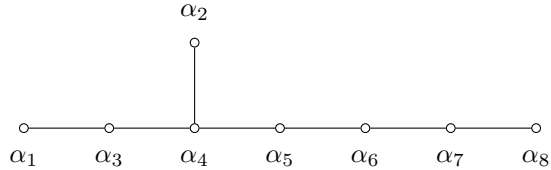
\begin{figure}[!h]%
\centering
\begin{tikzpicture}
    \node[main node] (1) {};
            \node[main node] (3) [right=1cm of 1] {};
            \node[main node] (4) [right=1cm of 3] {};
            \node[main node] (5) [right=1cm of 4] {};
            \node[main node] (6) [right=1cm of 5] {};
            \node[main node] (7) [right=1cm of 6] {};
            \node[main node] (8) [right=1cm of 7] {};
            \node[main node] (2) [above=1cm of 4] {};

            \path[draw]
            (1)--(3)--(4)--(5)--(6)--(7)--(8)
            (2)--(4);

            \node (11) [below=0.1cm of 1] {\small{$\al_1$}};
            \node (22) [above=0.1cm of 2] {\small{$\al_2$}};
            \node (33) [below=0.1cm of 3] {\small{$\al_3$}};
            \node (44) [below=0.1cm of 4] {\small{$\al_4$}};
            \node (55) [below=0.1cm of 5] {\small{$\al_5$}};
            \node (66) [below=0.1cm of 6] {\small{$\al_6$}};
            \node (77) [below=0.1cm of 7] {\small{$\al_7$}};
            \node (88) [below=0.1cm of 8] {\small{$\al_8$}};
\end{tikzpicture}
\caption{The Dynkin diagram of $\Phi$}%
\label{fig:dynkin}
\end{figure}

The root lattice $\Z\Pi$ has a natural partial order induced by the simple
roots, defined by the condition that $\al \leq \be$ if $\be - \al$ is a
nonnegative linear combination of simple roots. With respect to the restriction
of $\le$, the set $\Phi$ has a unique maximal element, $\theta$, called the
\emph{highest root}. We have $\theta=\mr_{89} =  e_9 - e_8$.

Each root $\al \in \allroots$ gives rise to a \emph{reflection}, which is the
self-inverse $\Z$-linear map $s_\alpha: \Z\Pi\ra \Z\Pi$ given by the formula
\begin{equation} \label{eq:refl} s_\al(\be) = \be - B(\al, \be)\al.
\end{equation} The set of all reflections $T=\{s_\al: \al\in \Phi\}$ generates
the \emph{Weyl group}, $W=W(E_8)$, and the action in Equation \eqref{eq:refl}
extends to a unique $W$-representation on $V$ called the \emph{reflection
representation}. The bilinear form $B$ is $W$-invariant in the sense that
$B(w\cdot \al,w\cdot \be)=B(\al,\be)$ for all $\al,\be\in \Phi$. 

The set of \emph{simple reflections}, $S := \{s_\al : \al \in \simproots\}$,
suffices to generate the Weyl group $W$, and the pair $(W,S)$ forms a Coxeter
system. 

We end this subsection by introducing some definitions that will be useful in
the sequel. 

\begin{defn}
\label{defn:abs}
For every root $\al\in \Phi$, we define $\abs{\al}$ to be the positive root
among the pair $\{\al,-\al\}$. For any $i,j,k\in \nine$, we define 
$r_{ij}=\abs{p_{ij}}=\abs{m_{ij}}$ and $r_{ijk}=\abs{p_{ijk}}=\abs{m_{ijk}}$.
\end{defn}

\begin{defn}
\label{defn:S9}
We define $s_{i}:=s_{\al_i}$ for all $i\in \{1,2,...,8\}$, and we set
$s_9:=s_\theta$, where $\theta=e_9-e_8$ is the highest root of $\Phi$. We let
$S'=(S\setminus\{s_2\})\cup\{s_9\}$, and we define $W'=\ip{S'}$, the subgroup of
$W$ generated by $S'$.
\end{defn}

\begin{rmk}
\label{rmk:S9}
It follows from the standard forms of the roots in $S'$ and Remark
\ref{rmk:orthog} that $W'$ is isomorphic to the symmetric group $S_9$ and acts
on $V$ by the obvious permutations on coordinates, with $s_1,s_3,s_4,...,s_8$,
and $s_9$ acting respectively as the adjacent transpositions
$(12),(23),(34),..., (78)$, and $(89)$ on the coordinate indices. The pair
$(W',S')$ forms a Coxeter system of type $A_8$. \end{rmk}

\subsection{Orthogonal bases}
\label{sec:2025}

We call a subset of the root system $\Phi=\Phi(E_8)$ \emph{positive} if its
elements are all positive roots, and \emph{orthogonal} if its elements are
pairwise orthogonal. We recall from \cite{gx5,gx6} some key definitions and
facts concerning positive orthogonal subsets of $\Phi$ in this subsection. Our
focus will be on the set 
\[
  \ob=\{R\se \Phi_+: \abs{R}=8 \text{ and $R$ is orthogonal}\}.
\] 
The elements of $\ob$ are called ``positive $8$-roots" in \cite{gx5,gx6}, and
they can be equivalently described as the maximal positive orthogonal subsets of
$\Phi$, or as the orthogonal bases of $V$ that consist of positive roots
\cite[Remark 2.4 (ii)]{gx6}. Moreover, the set $\ob$ can be identified with a
natural spanning set of a certain Macdonald representation of $W$ \cite[Section
3]{gx5}. 

The set $\ob$ admits a natural action of the Weyl group $W$, called the
\emph{standard action}, given by the formula 
\[ 
  w\cdot R=\{\abs{w\cdot \al}:\al\in R\} 
\] 
for all $w\in W$ and $R\in \ob$. Along with a suitably defined level function
$\lambda:\ob\ra \Z$, the standard action endows $\ob$ with the structure of a
special type of partially ordered set called a quasiparabolic $W$-set, in the
following sense of Rains--Vazirani: 

\begin{defn}
\cite[Section 2, Section 5]{rains13}
\label{def:qpset}
Let $(W,S)$ be a Coxeter system and $T$ its set of reflections. A {\it scaled
$W$-set} is a pair $(X, \lambda)$ where $X$ is a $W$-set and $\lambda : X
\rightarrow \Z$ is a function, called the \emph{level function}, such that
$|\lambda(sx) - \lambda(x)| \leq 1$ for all $s \in S$. A {\it quasiparabolic
set for $W$}, or \emph{quasiparabolic $W$-set}, is a scaled $W$-set $X$
satisfying the following two properties:
\begin{itemize}[leftmargin=3.2em] 
  \item[(QP1)]{for any $r \in T$ and $x \in X$,
    if $\lambda(rx) = \lambda(x)$, then $rx = x$;} 
  \item[(QP2)]{for any $r \in
T$, $x \in X$, and $s \in S$, if $\lambda(rx) > \lambda(x)$ and $\lambda(srx) <
\lambda(sx)$, then $rx = sx$.} 
\end{itemize} 
For a quasiparabolic set $X$, we define the \emph{quasiparabolic order} on $X$,
$\le_Q$, to be the weakest partial order such that $x \leq_Q rx$ whenever $x \in
X$, $r \in T$, and $\lambda(x) \leq \lambda(rx)$. The \emph{Poincar\'e
polynomial} of $X$ is the level generating function $PS_X(q)\in \Z[q]$ given by
\[
  PS_X(q)=\sum_{x\in X}q^{\lambda(x)}.
\]
\end{defn}

The level function $\lambda$ making $\ob$ a quasiparabolic $W$-set can be
defined as follows. We say a positive root $\be\in \Phi_+$ \emph{dominates}
another positive root $\gamma\in \Phi_+$ if $s_\gamma(\beta)\in \Phi_+$, and we
say that $\gamma$ \emph{topples} $\beta$ if $s_\gamma(\beta)$
is a negative root. We say $\be$ \emph{dominates} an orthogonal basis $R\in \ob$
if $\be$ dominates $\gamma$ for all $\gamma \in R$ (equivalently, if no element
of $R$ topples $\be$). We call the set
\[
  \res(R)=\{\gamma\in \Phi_+: \gamma \text{\; dominates\;} R\}
\]
the {\it (generalized) Rothe diagram} of $R$, and we define the  {\it level} of
$R$ to be the number $\rho(R) = |\res(R)|$. It follows from  \cite[Theorem
4.3]{gx6} that $W$ acts transitively on $\ob$ and that $(\ob,\rho)$ is a
quasiparabolic $W$-set with Poincar\'e polynomial 
\begin{equation}
\label{eq:poincare}
PS_\ob(q)=[3]_q[5]_q[9]_q[15]_q,
\end{equation}
where $[d]_q=\frac{q^d-1}{q-1}=1+q+\dots+q^{d-1}$ for each integer $d$. In
particular, we have $\abs{\ob}=3\cdot 5\cdot 9\cdot 15=2025$. Equation
\eqref{eq:poincare} was obtained via computer computation in \cite{gx5}, and one
main objective of this paper is to prove the equation without computer assistance.

\begin{rmk}
\label{rmk:self_topple}
Every positive root $\al\in \Phi_+$ topples itself because $s_\al(\al)=-\al<0$,
so $\res(R)\cap R=\emptyset$ for all $R\in \ob$. 
\end{rmk}

\begin{exa}
  \label{exa:minmax} 
Let 
\begin{equation*}
\label{eq:theta_A}
\theta_A:=\{p_{23}, p_{45}, p_{67}, m_{89}, m_{123}, m_{145}, m_{167}, p_{189}\}
\end{equation*}
and
\begin{equation*}
\label{eq:theta_N}
\theta_N:=\{p_{12}, p_{389}, p_{479}, p_{569}, m_{345}, m_{367}, m_{468}, m_{578}\} 
\end{equation*}
It follows from the general theory of quasiparabolic sets that, as a finite
transitive quasiparabolic $W$-set, the set $\ob$ contains a unique minimal
element and a unique maximal element with respect to the quasiparabolic order
$\le_Q$. More specifically, it follows from \cite[Proposition 4.11, Proposition
6.9]{gx5} that $\theta_A$ and $\theta_N$ are the unique minimal and maximal
elements of $\ob$, respectively.  \end{exa}

\subsection{Sum equivalence}
\label{sec:sums}
For each $R\in \ob$, we define the \emph{sum} of $R$ to be the element
\[
\sigma(R)=\sum_{\al\in R}\al\in \Z\Pi.
\]
We say $R, R'\in \ob$ are \emph{sum equivalent}, and we write $R\sim_\sigma R'$,
if $\sigma(R)=\sigma(R')$. This defines an equivalence relation on $\ob$ whose
equivalence classes are the fibres of $\sigma$. We will call each of these
equivalence classes a \emph{$\sigma$-class}.

We will compute the level generating function of each $\sigma$-class of $\ob$
and then use these generating functions to obtain the Poincar\'e polynomial of
$\ob$ in Section \ref{sec:gen_func}.  The goal of this subsection is to recall
from \cite{gx5} the relevant properties of the set $\Sigma$. We need the
following two definitions.

\begin{defn}
\label{defn:feature}
We define a \emph{feature} in $\Phi$ to be a set $Q=\{\be_1,\be_2,\be_3,\be_4\}$
of four pairwise orthogonal positive roots such that $\be_1+\be_2+\be_3+\be_4=2\be$ for
some $\be\in \Phi$. 
\end{defn}

It is known that for every feature $Q\se \Phi$ the set $\Psi_{Q}=\Z Q\cap \Phi$
is isomorphic to a root system of type $D_4$, and that $Q$ contains either $0,1$
or $3$ induced simple roots of $\Psi_Q$ \cite[Remark 2.14]{gx6}. These induced
simple roots can be characterized as the elements in $\Psi_+:=\Psi_Q\cap\Phi_+$
that cannot be written as a $\Z_+$-linear combination of other elements of
$\Psi_+$.

\begin{defn} \label{def:cna} 
  Let $Q$ be a feature in $\Phi$, and let $q$ be the number of induced simple
  roots of $\Psi_Q$ that $Q$ contains. We call $Q$ a \emph{crossing} if $q=0$, a
  \emph{nesting} if $q=1$, and an \emph{alignment} if $q=3$. We say $R$ is
  \emph{noncrossing} or \emph{nonnesting} if $R$ contains no crossings or no
  nestings, respectively. 
\end{defn}

\begin{rmk}
\label{rmk:theta_C}
\begin{enumerate}
\item It is possible to compute the level of each $R\in \ob$ by counting the
  crossings and nestings in $R$: by \cite[Corollary 4.2]{gx6}, we have
  $\rho(R)=C(R)+2N(R)$ for all $R\in \ob$, where $C(R)$ and $N(R)$ are the
  numbers of crossings and nestings in $R$, respectively.
\item By \cite[Proposition 4.11, Proposition 6.9]{gx5}, the unique minimal
  element $\theta_A\in \ob$ from Example \ref{exa:minmax} is the unique element
  of $\ob$ with no crossings and nestings; the unique maximal element $\theta_N$
  is the unique element with no alignments or crossings, and it has 14 nestings.
  Note that this implies that $\rho(\theta_A)=0$ and $\rho(\theta_N)=28$ by (i). 
\item By \cite[Proposition 5.2 (iii), Proposition 6.9]{gx5}, there is also a
  unique element in $\ob$, $\theta_C$, that contains no
  alignments or nestings, and $\theta_C$ is the unique maximal nonnesting
  element of $\ob$ with respect to $\le_Q$. It is given by the set
\begin{equation*}
\label{eq:theta_C}
\theta_C=\{m_{19}, m_{238}, m_{247}, m_{256}, m_{346}, m_{357}, m_{458}, m_{678}\}.
\end{equation*}
There are exactly 14 crossings in $\theta_C$ by \cite[Corollary 3.19]{gx5}, so
we have $\rho(\theta_C)=14$ by (i). 

\item We have 
  \begin{equation}
  \label{eq:theta_sums}
  \sigma(\theta_A)=4e_9+2(e_2+e_4+e_6+e_8) \quad \text{and }
  \sigma(\theta_C)=\sigma(\theta_N)=4e_9+2e_1.
  \end{equation}
These equations can be proved by direct computation and application of the
relation that $\sum_{i=1}^9e_i=0$. Alternatively, we will be able to deduce the
first equation immediately from Lemma \ref{lemm:mob_sum} and Example
\ref{exa:mob_exa}, and the second equation is immediate from Lemma
\ref{lemm:fob_sum} and Example \ref{exa:fob_exa}.  
\end{enumerate}
\end{rmk}

The following proposition summarizes from \cite{gx5} the key properties of the
set $\Sigma$ that are useful for this paper, including how $\Sigma$ is
compatible with the quasiparabolic structure on $\ob$, how $\Sigma$ can be
naturally indexed by the nonnesting or noncrossing elements of $\ob$, and how
the quasiparabolic order on $\ob$ induces the structure of a distributive
lattice on $\Sigma$. These properties will play an important role in the
arguments of Section \ref{sec:seq_to_sum}.

\begin{prop}
\label{prop:Sigma}
Let $\theta_A$ and $\theta_C$ be as in Remark \ref{rmk:theta_C}. Let
$\nn(\ob)$ be the set of all nonnesting elements in $\ob$.
\begin{enumerate} 
\item For any $R,R'\in \ob$, if $R\le_Q R'$ then we have $\sigma(R)\le
  \sigma(R')$.

\item For every $\gamma\in \Sigma$, the fibre $\sigma\inverse(\gamma)=\{R\in
  \ob:\sigma(R)=\gamma\}$ is an interval with respect to $\le_Q$, in the sense
  that it contains a unique minimal element, $m_\gamma$, and a unique maximal
  element, $M_\gamma$, such that $ \sigma\inverse(\gamma)=\{R\in
  \ob:m_\gamma\le_Q R\le_Q M_\gamma\}. $ Moreover, the element $m_\gamma$ is the
  unique nonnesting element of $\sigma\inverse(\gamma)$ and $M_\gamma$ is the
  unique noncrossing element of $\sigma\inverse(\gamma)$.

\item The map $\sigma:\nn(\ob)\ra\Sigma, R\mapsto \sigma(R)$ is a bijection.

\item We have $\wnn(\theta_C)=\theta_A$ for the element
  \begin{equation}\label{eq:w}
    \wnn=(s_1s_4s_6s_8)(s_3s_5s_7)(s_4s_6)(s_2s_5)(s_4)(s_3)(s_1)\in W. 
  \end{equation}
The element $\wnn$ is fully commutative \textup{(}in the sense of \cite{FC}\textup{)}, and for the set 
 \[
    \mathcal{J}=\{\nu\in W:\nu\le_L \wnn\inverse \} 
  \]
where $\le_L$ denotes the left weak Bruhat order on $W$, there is
  a bijection   
  \begin{equation}\label{eq:elt_to_sum} 
    \phi: \calj \ra \nn(\ob), \quad \nu\mapsto  \nu(\theta_A).
  \end{equation}
\item The set $\Sigma$ has 50 elements, and it has the structure of a
  distributive lattice. 
%
\end{enumerate}
\end{prop}

\begin{proof}
Part (i) follows from \cite[Corollary 4.9 (i)]{gx5}. Part (ii) follows from
\cite[Proposition 5.15 (iii)]{gx5}, and it implies (iii). Part (iv) follows from
\cite[Theorem 5.13 (i)--(ii)]{gx5}. We have $\abs{\nn(\ob)}=50$ by \cite[Theorem
5.16, Proposition 6.9]{gx5}, so (iii) implies that $\sigma$ has 50 elements as
well. The map $\sigma\circ\phi: \calj\ra \Sigma$ is a bijection by (iii) and
(iv), and $\mathcal{J}$ is well known to be a distributive lattice under the
weak order $\le_L$ (see, for example, \cite[Theorem 3.2]{FC}), so
$\sigma\circ\phi$ naturally induces a distributive lattice structure on
$\Sigma$. This proves (v). 
\end{proof}

\begin{rmk}
\label{rmk:isos}
\begin{enumerate}
\item 
The sets $\calj, \nn(\ob),$ and $\Sigma$ are subposets of $W, \ob$, and $\Z\Pi$
under the restrictions of the left weak order $\le_L$, the quasiparabolic order
$\le_Q$, and the order $\le$, respectively. Using the results from
\cite[Proposition 5.12, Theorem 5.13]{gx5} and properties of heaps, it can be
proved that with respect to the above partial orders, the bijections
$\phi:\calj\ra \nn(\ob)$ and $\sigma:\nn(\ob)\ra \Sigma$ are in fact poset
isomorphisms. In particular, since $\calj$ is a distributive lattice under
$\le_L$, the set $\nn(\ob)$ is a distributive lattice under $\le_Q$, and 
$\Sigma$  is a distributive lattice under $\le$. We will not need these
facts, however, so we omit their proofs for reasons of space.

\item The poset $\calj$ has $1_W$ and $\wnn\inverse$
  as its unique minimal and unique maximal element, respectively, by its
  definition. The fact that
  $\phi$ and $\sigma$ are poset isomorphisms thus implies that
  $1_W(\theta_A)=\theta_A$  
  and $\wnn\inverse(\theta_A)=\theta_C$ are the unique minimal and unique
  maximal elements of $\nn(\ob)$ (as noted in Remark \ref{rmk:theta_C}
(ii)--(iii)), and that the $\sigma$-classes of $\theta_A$ and $\theta_C$ are the
unique minimal and unique maximal $\sigma$-classes in $\ob$. More generally, we
can think of each element
 $\nu\in \calj$ as a raising operation on $\nn(\ob)$ that raises $\theta_A$ to
 a larger element (with respect to $\le_Q$) while also raising the sum (with
 respect to $\le$).

\item It can be shown, using \cite[Corollary 4.9, Proposition 5.12]{gx5} and
  properties of $\calj$, that $\sim_\sigma$ is a poset congruence of $\ob$ with
  respect to $\le_Q$ in the sense of \cite[Section 5]{reading02}. Furthermore,
  each congruence class is canonically isomorphic to a lower interval of the top
  class. In other words, as a poset under the restriction of $\le_Q$, every
  $\sigma$-class in $\ob$ is isomorphic to the interval $[\theta_C,R]:=\{R'\in
  \ob: \theta_C\le_Q R'\le_Q R\}$ for some $R\in \ob$ such that
$\sigma(R)=\sigma(\theta_C)$. \end{enumerate} \end{rmk}

The distributive lattice $\calj$ from Proposition \ref{prop:Sigma} can be
understood via ideals of a suitable poset $P$, as follows. Here, we recall that
an ideal of $P$ is a subset $I\se P$ such that if $y\in I$ and $x\in P$
satisfies $x\le y$ then $x\in I$, and that the set $J(P)$ of ideals of $P$ forms
a poset under set inclusion. By the theory of fully commutative elements
\cite[Section 2.2,  Theorem 3.2]{FC}, the element $\wnn$ from Proposition
\ref{prop:Sigma} gives rise to a poset, $P$,
whose Hasse diagram is as shown in Figure \ref{fig:heap} and whose elements are
naturally labelled by the generators appearing in the reduced word of $\wnn$ from
Equation \eqref{eq:w}. The poset $P$ is called the \emph{heap} of $\wnn$, and there is a natural isomorphism of posets 
\begin{equation}
\label{eq:ideal_to_elt}
\varphi: J(P)\ra \calj, \quad I\mapsto \varphi(I):=\nu_I
\end{equation}
where $\nu_I= \varphi(I)$ is the inverse of the unique element in $W$ whose
reduced words are
given by the linear extensions of $I$ of each $I\in J(P)$.

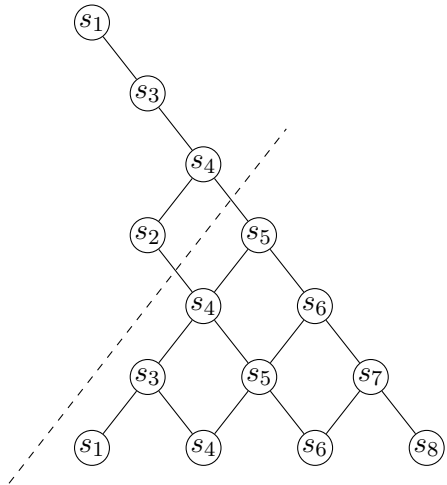
\begin{figure}[h!]
\begin{tikzpicture}
\node[main node] (0) {$s_1$};
\node[main node] (1) [right=1cm of 0] {$s_4$};
\node[main node] (2) [right=1cm of 1] {$s_6$};
\node[main node] (3) [right=1cm of 2] {$s_8$};
\node[main node] (4) [above right=0.6cm and 0.4cm of 0] {$s_3$};
\node[main node] (5) [above right=0.6cm and 0.4cm of 1] {$s_5$};
\node[main node] (6) [above right=0.6cm and 0.4cm of 2] {$s_7$};
\node[main node] (7) [above right=0.6cm and 0.4cm of 4] {$s_4$};
\node[main node] (8) [above right=0.6cm and 0.4cm of 5] {$s_6$};
\node[main node] (9) [above right=0.6cm and 0.4cm of 7] {$s_5$};
\node[main node] (10) [above left=0.6cm and 0.4cm of 7] {$s_2$};
\node[main node] (11) [above right=0.6cm and 0.4cm of 10] {$s_4$};
\node[main node] (12) [above left=0.6cm and 0.4cm of 11] {$s_3$};
\node[main node] (13) [above left=0.6cm and 0.4cm of 12] {$s_1$};

\node (x) [below left = 0.3cm and 0.9cm of 0] {};
\node (y) [above right = 0.3cm and 0.9cm of 11] {};

\path[draw]
(0)--(4)--(1)--(5)--(2)--(6)--(3)
(4)--(7)--(5)--(8)--(6)
(8)--(9)--(7)--(10)--(11)--(12)--(13)
(9)--(11);
\path[draw,dashed]
  (x)--(y);
\end{tikzpicture}
    \caption{The heap $P$.}
    \label{fig:heap}
\end{figure}

The heap $P$ can be conveniently partitioned into four chains, $C_1, C_2, C_3$,
and $C_4$, given by the four Northwest--Southeast diagonal chains from left to
right shown in Figure \ref{fig:heap}. The lengths of these chains are $c_1=1,
c_2=2, c_3=4,$ and $c_4=7$, respectively, and it follows from the Hasse diagram
that the set $J(P)$ is in a natural bijection with the sequences in the set 
\begin{equation}
\label{eq:seqs}
\cals=\{(d_1,d_2,d_3,d_4): 0 \le d_i\le c_i \text{ for all $1\le i\le 4$ and
$d_{i-1}\ge \min(d_i-1,c_{i-1})$ for all $2\le i \le 4$}\}
\end{equation}
via the map 
\begin{equation}
\label{eq:seq_to_ideal}
\psi: \cals \ra J(P), \quad d\mapsto I_d,
\end{equation}
where each sequence $\seqd:=(d_1,d_2,d_3,d_4)$ corresponds to the unique ideal
$I_{\seqd}$
containing exactly $d_i$ elements from $C_i$ for all $i$. The set $\cals$ naturally forms a poset under entrywise
comparison, and the bijection $\psi$ is readily seen to be a poset isomorphism
with respect to this partial order and the inclusion order on $J(P)$.

Combining the poset isomorphisms $\psi$ and $\varphi$ with the poset isomorphisms
$\phi$ and $\sigma$ from Remark \ref{rmk:isos} (i) yields the following 
useful isomorphism.

\begin{defn}
\label{defn:iso_chain}
We define $\Psi: \cals\ra \Sigma$ to be the following composition of four
bijections:
\begin{equation}
\label{eq:iso_chain}
 \Psi:= \sigma\circ \phi\circ \varphi\circ \psi: \cals\stackrel{\psi}{\ra} J(P) \stackrel{\varphi}{\ra} \calj
  \stackrel{\phi}{\ra} \nn(\ob)\stackrel{\sigma}{\ra} \Sigma.
\end{equation}
\end{defn}

\begin{exa}
\label{exa:bijections}
Let $\seqd=(1,2,3,4)\in \cals$. The map $\psi:\cals\ra J(P)$ sends $\seqd$ to the ideal
$I_{\seqd}\in J(P)$ consisting of all elements below the dashed line in Figure 2. This
ideal is in turn sent by the map $\varphi: J(P)\ra \calj$ to the element 
  \[ 
    w_{\seqd}:=\varphi(I_{\seqd})=  ((s_1)(s_4s_3)(s_6s_5s_4)(s_8s_7s_6s_5))\inverse=(s_5s_6s_7s_8)(s_4s_5s_6)(s_3s_4)(s_1)\in
W. \] 
We will explain how to compute the sum $\Psi(\seqd)=\sigma(w_{\seqd}(\theta_A))$
directly from $\seqd$ in
Proposition \ref{prop:seq_to_sum} (i), and it will follow immediately that  
\begin{equation}
\label{eq:max_fob_sum}
\Psi(\seqd)
=4e_9+2(e_1+e_2+e_3+e_4). 
\end{equation}
\end{exa}

The isomorphism $\Psi$ allows us to understand the elements of $\Sigma$ using
the sequences in $\cals$ and understand the order structure of $\Sigma$ in terms
of the natural partial order on $\cals$ given by entrywise comparison. The
sequences in $\cals$ also lead to the following concise formula for the level
generating function of the $\sigma$-classes.

\begin{theorem}
\label{thm:gen_func}
Let $\seqd=(d_1,d_2,d_3,d_4)\in \cals, \gamma=\Psi(\seqd)\in \Sigma$, and
$\abs{\seqd}=d_1+d_2+d_3+d_4$. We have
\[
  \sum_{R\in \sigma\inverse(\gamma)}q^{\rho(R)}=q^{\abs{\seqd}}\prod_{i=1}^4 
  [d_i+1]_q.
\]
\end{theorem}

Theorem \ref{thm:gen_func} is the main enumerative result of this paper. We will
prove it and use it to derive the Poincar\'e polynomial of $\ob$ in Section
\ref{sec:gen_func}. Before we do that, we will develop two useful combinatorial
models for the orthogonal bases of $\ob$ and their Rothe diagrams in the next two
sections. These models involve certain perfect matchings and labellings of the
Fano plane, and the results of the next two sections will allow us to understand
the $\sigma$-classes more thoroughly, on the level of the entire classes
$\sigma\inverse(\gamma)$, rather than just the minimal elements $m_\gamma\in
\nn(\ob)\cap \sigma\inverse(\gamma)$ that represent the classes.

\section{Orthogonal bases via Perfect Matchings}
\label{sec:mob}

In this section, we explain how we can use the perfect matchings of the set
$\ten = \{0, 1, ..., 9\}$ to understand a certain subset of $\ob$ of
size $945$. We set up the necessary definitions and describe our construction in
Section \ref{sec:mob_def}, characterize the Rothe diagrams of the bases via the
perfect matchings in Section \ref{sec:mob_res}, and explain how the levels of
the bases can be computed directly in terms of the matchings in Section
\ref{sec:mob_level}.

\subsection{Definitions}
\label{sec:mob_def}
Recall that for a finite set $Z$, a \emph{perfect matching} of $Z$ is a partition of
the set into subsets of size 2. Such a partition exists if and only
if the cardinality of $Z$ is an even number $2k$, and in this case 
$Z$ has a total of $(2k-1)!!=(2k-1)(2k-3)\cdots 3\cdot 1$
perfect matchings. In particular, the set $\ten$ has $9!!=945$ perfect matchings. 
We will use $S_Z$ to denote the symmetric group on
$Z$ throughout the section.

\begin{defn}
\label{defn:mdef}
Let $k$ be a positive integer, and let $Z\se \Z$ be a finite set of cardinality $2k$.
\begin{enumerate}
  \item We write each {perfect matching} of $Z$ in the form 
    \begin{equation}
\label{eq:matching}
M = \{(i_1, j_1), \ (i_2, j_2), \dots, \ (i_k, j_k)\}
\end{equation}
where $\{i_1,j_1,i_2,j_2,\dots, i_k,j_k\}=Z$ and $i_a<j_a$ for all $1\le
a\le k$. We call each pair $p=(i,j)\in M$ an \emph{arc} in $M$, and we call $i$
and $j$ the \emph{left endpoint} and \emph{right endpoint} of $p$ (or of $M$),
respectively. We say that $p$ (or $M$) \emph{connects} or \emph{pairs} $i$ and
$j$, and we call $p$ a \emph{middle arc} if $i\neq \min(Z)$ and $j\neq \max(Z)$. 

\item For each perfect matching $M$ of $Z$, we define $\tau=\tau_M$ to be the unique
  fixed-point-free involution in $S_Z$ that sends each element $i\in Z$ to
  the unique number paired with $i$ by $M$. 

\item We denote the set of all perfect matchings of $\ten$ by $\calm$. For each
$M\in \calm$, we say $M$ is of \emph{type I} if $\tau(0)=9$ and of \emph{type II}
otherwise. 
\end{enumerate} 
\end{defn}

The 945 perfect matchings of $\ten$ give rise to $945$ elements of $\ob$ as
follows.

\begin{defn}
\label{defn:mob_construction}
For each $M=\{(i_a,j_a):1\le a\le 5\}\in \calm$, we will henceforth
  assume, without loss of generality, that $i_5=0$, and we define  \[ R(M) =
    \{\rr_{i_1 j_1}, \ \rr_{i_2 j_2}, \ \rr_{i_3 j_3}, \ \rr_{i_4 j_4}, \
      \rr_{i_1 j_1z}, \ \rr_{i_2 j_2z}, \ \rr_{i_3 j_3z}, \ \rr_{i_4 j_4z}\}, \]
    where $z=j_5=\tau(0)$. We define \[ \mob=\{R(M):M\in \calm\}. \]
\end{defn}

\begin{lemma}
\label{lemm:mob_sum}
For every perfect matching $M=\{(i_1, j_1), \ (i_2, j_2), \ (i_3, j_3), \ (i_4,
j_4), \ \{0, z\}\}$, we have $R(M)\in\ob$ and 
\begin{equation}
  \label{eq:mob_sum}
  \sigma(R(M)) = 4e_9 + 2(e_{i_1} + e_{i_2} + e_{i_3} + e_{i_4}).
\end{equation}
The map $f: \calm\ra \mob, M\mapsto R(M)$ is bijective, so that
$\abs{\mob}=945$.
\end{lemma}

\begin{proof}
It follows from Definition \ref{defn:mob_construction} and Remark
\ref{rmk:orthog} that $R(M)$ consists of 8 distinct and pairwise orthogonal
positive roots, so we have $R(M)\in \ob$. The surjectivity of the map $f$ then
follows from the definition of $\mob$, and $f$ is injective because the roots
$r_{i_aj_a}\in R(M)$ allow us to recover $M$. It follows that
$\abs{\mob}=\abs{\calm}=9!!=945.$

It remains to prove Equation \eqref{eq:mob_sum}. If $M$ is of type I, then
Equation \eqref{eq:mob_sum} follows from direct computation using the standard
form of the roots in $R(M)$. If $M$ is of type II, then we may assume that
$j_4=9$, in which case direct computation shows that
$\sigma(R(M))=-2(e_{j_1}+e_{j_2}+e_{j_3}+e_z)+2e_9$. Equation \eqref{eq:mob_sum}
then follows from the relation $\sum_{i=1}^9e _i = 0$. \end{proof}

\begin{exa}
\label{exa:mob_exa}
The matching $M_A=\{(0,1), (2,3), (4,5), (6,7), (8,9)\}$ shown in Figure
\ref{fig:matchings} (a) gives rise to the orthogonal basis 
\[
R(M_A)= \{p_{23}, p_{45}, p_{67}, m_{89}, m_{123}, m_{145}, m_{167}, p_{189}\},
\]
which is precisely the unique minimal $\theta_A$ of $\ob$ given in Example
\ref{exa:minmax}. The matching shown
in Figure \ref{fig:matchings} (b), given by $M_C=\{(0,5), (1,6), (2,7), (3,8),
(4,9)\}$, corresponds 
to the orthogonal basis
\[
R(M_C)= \{p_{16}, p_{27}, p_{38}, m_{49}, m_{156}, m_{257}, m_{358}, p_{459}\}.
\]
It will follow from the results of Section \ref{sec:seq_to_sum} that
$M_C=w_{\seqd}(M_A)$ and $R(M_C)=w_{\seqd}(\theta_A)$, where $w_{\seqd}\in \calj$ is the element
corresponding to the sequence $\seqd=(1,2,3,4)\in \cals$ as in Example
\ref{exa:bijections}. Both the matchings $M_A$ and $M_C$ are of type II. We note
that while $R(M_A)=\theta_A$, the orthogonal basis $R(M_C)$ is not equal to
$\theta_C$. The element that will naturally give rise to $\theta_C$ is instead
another element, $F_C$, as we will see in Example \ref{exa:fob_exa}.

\end{exa}

\begin{figure}[h!]
  \centering
  \subfloat[]
  {
\begin{tikzpicture}[scale=1]
    \foreach \x in {0,...,9} {
        \fill (\x,0) circle (1pt);
        \node[above] at (\x,0) {\x};
    }

    \foreach \x in {0,2,4,6,8} {
      \draw (\x,0) to[out=-90,in=-90,looseness=1.5] (\x+1,0);
    }

    \draw (-0.5,0) to (9.5,0);
\end{tikzpicture}
}
\quad
\subfloat[]
{
\begin{tikzpicture}[scale=1]

    \foreach \x in {0,...,9} {
        \fill (\x,0) circle (1pt);
        \node[above] at (\x,0) {\x};
    }


    \foreach \x in {0,...,4} {
      \draw (\x,0) to[out=-60,in=-120,looseness=1] (\x+5,0);
    }
    \draw (-0.5,0) to (9.5,0);
  \end{tikzpicture}
}
\caption{Two perfect matchings of $\ten$}
\label{fig:matchings}
\end{figure}
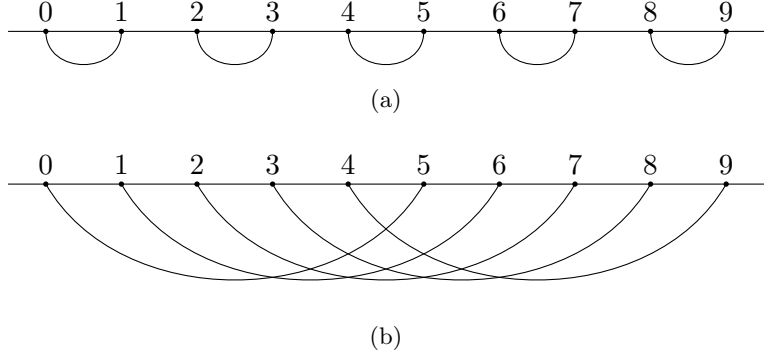

The next definition will allow us to describe the level $\rho(R(M))$ of $R(M)$
directly in terms of the matching $M$ in Section \ref{sec:mob_level}, which will
eventually further allow us to compute the level generating functions of the
$\sigma$-classes in $\mob$ in Section \ref{sec:gen_func}. Note that although
Part (i) of the definition uses the term ``feature'' again, this term refers to
pairs of arcs in the context of the new definition and thus differs from the
features defined in Definition \ref{defn:feature}, which refer to certain
quadruples of roots. Our choice to repeat the terminology is intentional and
will not cause ambiguity anywhere in the paper, and the connection between the
two notions of feature will be made precise in Remark \ref{rmk:features}. The
same comment applies to the notions of crossings, nestings, and alignments.

\begin{defn}
\label{def:matching_features}
Let $k$ be a positive integer and let 
\[
  M=\{(i_1,j_1), (i_2,j_2), \dots, (i_k,j_k)\}
\] 
be a perfect matching of a set $Z$ of $2k$ distinct integers. 
\begin{enumerate}
\item 
We call each pair of elements $F=\{(a,b), (c,d)\}\subseteq M$  a \emph{feature}. If
$a<c$, then we say that the feature $F$ is a \emph{crossing} if $a<c<b<d$, a
\emph{nesting} if $a<c<d<b$, and an \emph{alignment} if $a<b<c<d$; see Figure
\ref{fig:acn}.

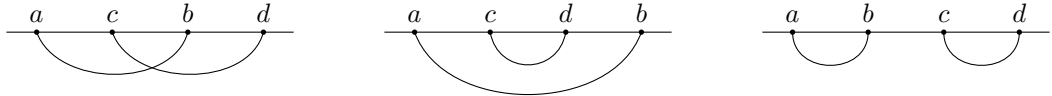
\begin{figure}[h!]
    \centering
  \begin{tikzpicture}[baseline=0pt]
    \foreach \x in {0,...,3} {
        \fill (\x,0) circle (1pt);
    }

      \node[above] at (0,0) {\small$a$};
      \node[above] at (1,0) {\small$c$};
      \node[above] at (2,0) {\small$b$};
      \node[above] at (3,0) {\small$d$};
    \foreach \x in {0,1} {
      \draw (\x,0) to[out=-70,in=-105,looseness=1] (\x+2,0);
    }

    \draw (-0.4,0) to (3.4,0);
\foreach \x in {0,...,3} {
        \fill (\x+5,0) circle (1pt);
    }

      \node[above] at (5,0) {\small$a$};
      \node[above] at (6,0) {\small$c$};
      \node[above] at (7,0) {\small$d$};
      \node[above] at (8,0) {\small$b$};

      \draw (5,0) to[out=-70,in=-110,looseness=1] (8,0);
      \draw (6,0) to[out=-80,in=-100,looseness=1.5] (7,0);

    \draw (5-0.4,0) to (8.4,0);
\foreach \x in {0,...,3} {
        \fill (\x+10,0) circle (1pt);
    }

      \node[above] at (10,0) {\small$a$};
      \node[above] at (11,0) {\small$b$};
      \node[above] at (12,0) {\small$c$};
      \node[above] at (13,0) {\small$d$};
 \foreach \x in {0,2} {
      \draw (\x + 10,0) to[out=-90,in=-90,looseness=1.5] (\x+ 11,0);
    }
    \draw (10-0.4,0) to (13.4,0);
\end{tikzpicture}
\caption{Crossing, nesting, and alignment (from left to right)}
\label{fig:acn}
\end{figure}

\item We define $c(M), n(M),$ and  $a(M)$ to be the number of crossings,
  nestings, and alignments contained in $M$, respectively. We define the
  \emph{height} of $M$ to be \[h(M)=c(M)+2n(M).\] 

\item We define the \emph{interior} of each arc $(i,j)\in M$ to be the set 
  $\interior(i,j)=\{k\in Z: i<k<j\}$, and define the \emph{weight} of $M$ to be
  \begin{equation} \label{eq:wt} 
    \wt(M)=\sum_{(i,j)\in M} \abs{\interior(i,j)}-c(M).
  \end{equation}

\item We define 
  \[
    \lambda(M)=\frac{\ell(\tau_M)-k}{2},
  \]
  where $\ell$ is the Coxeter length of $\tau_M$ (i.e., the number of inversions
  of $\tau_M$, where an inversion is a pair of numbers $i<j$ in $Z$ such that
  $\tau_M(i)>\tau_M(j)$).

\item Following \cite{hamaker18}, we define
  the \emph{involution Rothe diagram} of $\tau_M$ to be
  the set
  \[
    D_\texttt{FPF}(M)=\{(i,j) : i,j\in Z, i>j, \tau_M(i)>j, \tau_M(j)>i\}.
  \]
  We define 
  \[ 
  D_2(M)=\{(i,j) : i,j\in Z, i<j, \tau_M(i)<j, \tau_M(j)<i\},
  \]
  we set $d_1(M)=\abs{D_\texttt{FPF}(M)}$, and we set $d_2(M)=\abs{D_2(M)}$.
\end{enumerate}
\end{defn}

\begin{defn}
\label{defn:standardize}
Let $n\in \Z_{>0}$ and $[n]=\{1,2,\dots ,n\}$. For any set $Z$ of $n$ distinct
integers, we define the \emph{standardization of $Z$} to be the map $\pi: Z\ra
[n]$ that sends the $i$-th smallest element of $Z$ to $i$ for all $i\in [n]$. If
$n$ is even and $M=\{(i_a,j_a): 1\le a\le n/2\}$ is a perfect matching of $Z$,
then we define the \emph{standardization of $M$} to be the matching
$M'=\{(\pi(i_a),\pi(j_a)):1\le a\le n/2\}$ of $[n]$. \end{defn}

\begin{rmk}
\label{rmk:standardize}
Standardization of a set $Z\se \Z$ respects the relative order of 
the elements of $Z$. Therefore, if $M$ 
is a perfect matching of $Z$ and 
$M'$ is the standardization of $M$, then $M$ and $M'$ share the
same values under the functions $c,n,a,h, d_1$, and $d_2$ defined in Definition
\ref{def:matching_features}, because these functions depend only on the relative order of the elements involved in features. For
the standardization $M'$, the number $\abs{\interior(i,j)}$ for an arc $(i,j)\in
M'$ recovers Watson's definition of the \emph{span} of $(i,j)$, which is equal
to $j-i-1$, and $\wt(M')$ recovers the \emph{weight} function of Deodhar and
Srinivasan \cite{deodhar}. 
\end{rmk}

We will prove in Section \ref{sec:mob_level} that $\rho(R(M))=h(M)$ for each
$M\in \calm$, so that we can compute the level of $R(M)$ by counting crossings
and nestings of $M$. The following result gives further interpretations of the
height $h(M)$.

\begin{prop}
\label{prop:wt}
In the setting of Definition \ref{def:matching_features}, we have
\[
h(M)=\wt(M)=\lambda(M)=d_1(M)=d_2(M).
\]
\end{prop}

\begin{proof}
By Remark \ref{rmk:standardize}, we may assume $Z$ is standardized and $M$ is a
perfect matching of the set $[2k]$. In this case, the asserted equalities appear
well known to experts, but we will fill in some details that we have not been
able to find in the literature. 

To relate $\wt(M)$ to $h(M)$, let $F=\{(a,b), (c,d)\}$ be a crossing or nesting
in $M$ with $a<c$, and let $N=\sum_{(i,j)\in M}\abs{\interior(i,j)}$. If $F$ is
a crossing, then $c \in \interior(a,b)$ and $b\in \interior(c,d)$, so $F$
contributes a count of exactly 2 towards the sum $N$. If $F$ is a nesting, then
$c$ and $d$ lie in $\interior(a,b)$ and $F$ also contributes a count of $2$ towards
$N$, corresponding to the fact that $c$ and $d$ are interior points of $(a,b)$.
It follows that $N=2(c(M)+n(M))$, which implies that $\wt(M)=N-c(M)=h(M)$. 

The fact that $\wt(M)=\lambda(M)$ follows immediately from \cite[Proposition
9]{watson14}. 

It follows from the first paragraph of \cite[Section 3]{hamaker18} and
\cite[Proposition 3.6]{hamaker18} that $\lambda(M)$ equals the cardinality of
$D_\texttt{FPF}(M)$, so we have $\lambda(M)=d_1(M)$.

Finally, if we let $f$ and $g$ be the maps that sort each pair of distinct integers $\{a,b\}$ where $a<b$
to $(a,b)$ and $(b,a)$, respectively, then a routine
verification shows that the maps $\eta: D_\texttt{FPF}(M)\ra D_2(M), (i,j)\mapsto
f(\{\tau_M(i),\tau_M(j)\})$ and $\eta': D_2(M)\ra D_\texttt{FPF}(M), (i,j)\mapsto
g(\{\tau_M(i),\tau_M(j)\})$ are well-defined mutual inverses of each other, so
we have $d_1(M)=d_2(M)$. This completes the proof.
\end{proof}

\subsection{Rothe diagrams}
\label{sec:mob_res}
We explain how to use the perfect matchings $M\in \calm$ to compute the
Rothe diagram of the orthogonal bases $R(M)$ in this section.
Throughout the section, we will denote the root $s_\al(\be)$ by $\al(\be)$ for
any $\al,\be\in \Phi$.

\begin{lemma}
\label{lem:res_prep}
Let $\al\in \Phi_+$ and $R\in \ob$.
Suppose that $\nine=\{i,j,k,i',j',k',a,b,c\}$.
\begin{enumerate}
\item If $\al\in \res(R)$, then $\al\notin R$, and the set
  \[
    \pra:=\{\beta\in R: B(\al,\beta)\neq 0\}
  \]
  has size 4.  
\item We have $r_{ij'k'}(p_{ij})=m_{jj'k'}$ and $r_{i'jk'}(p_{ij})=p_{ii'k'}$.
\item We have $r_{ij'k'}(m_{ij})=p_{jj'k'}$ and $r_{i'jk'}(m_{ij})=m_{ii'k'}$.
\item We have $r_{ijk'}(p_{ijk})=p_{kk'}$.
\item We have 
  \[
    r_{i'j'k'}(r_{ijk})= 
  \begin{cases} 
    r_{i'j'k'}(p_{ijk})=m_{abc}>0 & \text{if $9\in \{i,j,k\}$};\\
    r_{i'j'k'}(m_{ijk})=p_{abc}<0 & \text{if $9\in \{i',j',k'\}$};\\
    r_{i'j'k'}(m_{ijk})=p_{abc}>0 & \text{if $9\notin \{i,j,k,i',j',k'\}$}.
  \end{cases}
  \]
\end{enumerate}
\end{lemma}

\begin{proof}
  If $\al\in \res(R)$, then we have $\al\notin R$ by Remark
  \ref{rmk:self_topple}. The fact that $\pra$ has size 4 then follows directly
  from \cite[Proposition 2.1 (i)]{gx5}. 

  Since $s_\be=s_{(-\be)}$ for any $\be\in \Phi$, we may pretend that in the
  proofs of (ii)--(v) that all the roots appearing in the form $r_{xyz}$ are
  $p_{xyz}$, even when $p_{xyz}<0$. The conclusions of (ii) and (iv) then follow
  readily from Remark \ref{rmk:orthog} and direct computation. For example,
  since $\abs{\{i,j,k\}\cap\{i,j,k'\}}=2$, Remark \ref{rmk:orthog} implies that
  $B(p_{ijk},p_{ijk'})=2-1=1$, so 
  \begin{equation}
  \label{eq:typical}
    r_{ijk'}(p_{ijk})=p_{ijk'}(p_{ijk})=p_{ijk}-1\cdot p_{ijk'}=p_{kk'}
  \end{equation}
  and (iv) holds. 

  Combining computations similar to Equation \eqref{eq:typical} and the relation
  $\sum_{p\in \nine}e_p=0$ shows that $r_{i'j'k'}(p_{ijk})=m_{abc}$ and
  $r_{i'j'k'}(m_{ijk})=p_{abc}$, and the conclusions of (v) follow. For example,
  in the third case, where $9\notin \{i,j,k,i',j',k'\}$ and therefore $9\in
  \{a,b,c\}$, we have $r_{ijk}=m_{ijk}$ since $9\notin \{i,j,k\}$, and
  $B(p_{i'j'k'},m_{ijk})=-B(p_{i'j'k'},p_{ijk})=1$, so we have 
  \[
    r_{i'j'k'}(r_{ijk})=p_{i'j'k'}(m_{ijk})=m_{ijk}-p_{i'j'k'}=m_{ijk}+m_{i'j'k'}=p_{abc}>0.
  \]
  Finally, (iii) follows from (ii) by linearity of the reflections. 
\end{proof}

\begin{theorem}
\label{thm:mob_res}
Let $M=\{(i_1,j_1), (i_2,j_2), (i_3,j_3), (i_4,
j_4),(0,z)\}\in \calm$, let $\tau=\tau_M$, and let $R=R(M)\in \ob$.
\begin{enumerate}
\item If $M$ is of type I, then  we have
$\res(R)=C_1 \cup C_2$,
where 
  \[
    C_1=\{m_{i9}: 1\le i\le 8\} \text{\; and\; } C_2=  \{p_{ij9}: 1\le i<j\le 8,
    \tau(i)>j, \tau(j)>i\}. 
  \]
\item If $M$ is of type II and $y=\tau(9)$, then we have $\res(R)=C_1\cup
  C_2\cup C_3$,  where
  \begin{align*} 
  C_1&=\{m_{i9}: y<i\le 8\},\\
  C_2&=\{p_{iy9}: 1\le i<z, i\neq y\},\text{ and}\\
  C_3&=\{m_{ijz}: i, j\notin \{0,9,y,z\}, i<j, \tau(i)<j,\tau(j)<i\}. 
  \end{align*}
\end{enumerate}
\end{theorem}

\begin{proof}
  Lemma \ref{lem:res_prep} (i) implies that if
  $\al\in \res(R)$, then $\al\notin R$ and only four elements in $R$,
  namely those in the set $\pra=\{\be\in R: B(\al,\be)\neq0\}$, have associated
  reflections that could possibly topple $\al$. Therefore, to
  find $\res(R)$ it suffices to characterize the positive roots $\al$ not in $R$
  such that $\beta(\al)>0$ for all $\be\in \pra$. Recall from \eqref{eq:pos}
  that any positive root has one of four possible standard forms: $p_{ij}$ for some
  $1\le i<j\le 8$, or $m_{i9}$ for some $1\le i \le 8$, or $p_{ij9}$ for some
  $1\le i<j\le 8$, or $m_{ijk}$ for some $1\le i<j<k\le 8$.

  Assume that $M$ is of type I, so that we have $z=\tau(0)=9$ and we can write
  $R$ in the form \[ R(M)=\{p_{i_1j_1}, p_{i_2j_2}, p_{i_3j_3}, p_{i_4j_4},
  p_{i_1j_19}, p_{i_2j_29}, p_{i_3j_39}, p_{i_4j_49}\}. \] We will find the
  elements $\al\notin R$ such that $\be(\al)>0$ for all $\be\in \pra$ using
  Table \ref{tab:mob1}. The first column of the table shows the standard form of
  $\al$ as recalled in the last paragraph, where we assume $i<j$ whenever $i$
  and $j$ appear together in a subscript. Note that the assumption that
  $\al\notin R$ places further constraints on $\al$, such as the condition that
  $j\notin \tau(i)$ if $\al=p_{ij}$.
  In the last two columns, where for each an
  \emph{$i$-excluding} arc ($i\in \ten$) means an arc that does not contain $i$,
  we either identify one element $\be\in \pra$ that topples $\al$ in the second
  column, in which case we display the element $\be(\al)<0$ in the third column,
  or we describe all the four elements $\be\in\pra$ in the second column and
  write down the four respective elements $\be(\al)$ in the third column.   In
  the first case, we may conclude immediately that $\al\notin \res(R)$; in the
  second case, we follow each root $\be(\al)$ with ``$>0$" if it is guaranteed
  to be positive by inspection, so that we have $\al\in \res(R)$ if and only if
  the remaining roots of the form $\be(\al)$ are positive. The fact that the
  elements $\be$ in the second column are indeed in $\pra$ follows from direct
  inspection, Definition \ref{defn:abs}, and Remark \ref{rmk:orthog}.
  The assertions indicated in the third column can be verified in one of two
  ways, as follows. If $\be$ is of the form $r_{ij}$, then we can compute
  $\be(\al)$ using the fact that $s_\be$ simply acts as the transposition $(ij)$
  on coordinate indices by Remark \ref{rmk:S9}, and then we can determine the
  sign of $\be(\al)$ using Equation \eqref{eq:pos}. If $\be$ is of the form
  $r_{i'j'k'}$, then the sign of $\be(\al)$ can be determined using the formulas
  in Lemma \ref{lem:res_prep} (ii)--(v). For example, the last row of the table
  lists an element $\be=p_{i'j'9}\in\pra $ that topples $\al$ in the second
  column, and the assertion that $\be(\al)<0$ follows immediately from Lemma
  \ref{lem:res_prep} (v).

\begin{table}[h!]
  \resizebox{\textwidth}{!}{
\begin{tabular}{@{}lll@{}}
\toprule
  $\al$ & $\be\in \pra$ & $\be(\al)$ \\
  \midrule
  $p_{ij}, j\neq \tau(i)$ & $p_{i\tau(i)9}$  & $m_{j\tau(i)9}<0$ \\
  \midrule
  $m_{i9}$ & $r_{i\tau(i)}$, or $p_{j\tau(j)9}$ for the $i$-excluding arcs $\{j,\tau(j)\}$&
  $m_{\tau(i)9}>0$, or $m_{ij\tau(j)}>0$\\
  \midrule
           $p_{ij9}, j\neq \tau(i)$& $r_{i\tau(i)}, r_{j\tau(j)}, p_{i\tau(i)9},
  p_{j\tau(j)9}$   & $p_{j\tau(i)9}>0, p_{i\tau(j)9}>0, p_{j\tau(i)}, p_{i\tau(j)}$ \\ 
  \midrule
   $m_{ijk}$ & $p_{i'j'9}$ where $(i',j')$ is any arc such that $i',j'\notin \{i,j,k\}$  &  $p_{i'j'9}(m_{ijk})<0$ \\
   \bottomrule
\end{tabular}
}
\caption{Analysis of $\be(\al)$, I}
\label{tab:mob1}
\end{table}

The first and last rows of Table \ref{tab:mob1} imply that $\res(R)$ contains no
root of the form $p_{ij}$ or $m_{ijk}$, the second row implies that every root
of the form $m_{i9}$ where $1\le i\le 8$ lies in $\res(R)$, and the third row
implies that a root of the form $p_{ij9}$  where $1\le i<j\le 8$ lies in
$\res(R)$ if and only if $p_{j\tau(i)}$ and $p_{i\tau(j)}$  are positive roots,
which occurs if and only if $j<\tau(i)$ and $i<\tau(j)$. This proves (i). 

  Now assume that $M$ is of type II, $\tau(0)=z$, and $\tau(9)=y$, so that we
  can write $M$ in the form 
   \[
    R(M)=\{m_{y9}, p_{i_1j_1}, p_{i_2j_2}, p_{i_3j_3}, p_{9yz}, m_{i_1j_1z},
    m_{i_2j_2z}, m_{i_3j_3z}\}.
  \]
  We can find the elements of $\res(R)$ using Table \ref{tab:mob2}, which is a
  direct analogue of Table \ref{tab:mob1} except that each possible standard
  form of $\al$ now corresponds a block of lines representing a complete, mutually
  exclusive list of possible subcases. 

The first block of Table \ref{tab:mob2} (which is shown in landscape mode in
Appendix \ref{sec:wide_tables} due to its width) implies that $p_{ij}\notin\res(R)$ for any $1\le
i<j\le 8$. The last three blocks imply that if $\al$ is of the form $m_{i9},
p_{ij9}$, or $m_{ijk}$, then $\al$ dominates $R$ if and only if it lies in the
set $C_1, C_2$, or $C_3$ from (ii), respectively. This completes the proof.
\end{proof}

\subsection{Levels}
\label{sec:mob_level}

The explicit description of the Rothe diagrams in Theorem \ref{thm:mob_res}
allows us to compute the level of $R(M)$ as the height of $M$ for each $M\in
\calm$, as follows.

\begin{prop}
\label{prop:mob_level}
We have $\rho(R(M))=h(M)$ for all $M\in \calm$.
\end{prop}

\begin{proof}
If $M$ is of type I, then $M$ consists of the arc $A=(0,9)$ and four middle arcs
that form a perfect matching, $M'$, of 
the set $\{1,2,\dots, 8\}$. Each of these middle arcs forms a nesting with $A$,
and it follows that we have 
\[
h(M)=c(M)+2n(M)=c(M')+2(n(M')+4)=h(M')+8.
\]
On the other hand, for the sets $C_1$ and $C_2$ defined as in Theorem
\ref{thm:mob_res} (i), it follows by inspection that we have
$\abs{C_1}=8$ and $\abs{C_2}=d_1(M')$, which implies that
\[
\rho(R(M))= 8 + d_1(M').
\]
Since $h(M')=d_1(M')$ by Proposition \ref{prop:wt},  it follows that
$\rho(R(M))=h(M)$.

Now suppose that $M$ is of type II, let $y=\tau_M(9)$, and let $z=\tau_M(0)$.
Let $A_9=(y,9)$,
let $A_0=(0,z)$, and let $M'$ be the perfect matching of
the set $(\ten\setminus\{0,9,y,z\})$ formed by the remaining three arcs in $M$. 
If $F$ is a nesting or crossing in $M$, then $F$ satisfies one of
the following mutually exclusive conditions: 
\begin{enumerate}
  \item[(a)] both arcs in $F$ are $M'$, so that $F$ is a feature of $M'$;
  \item[(b)] one of the arcs in $F$ is $A_9$;
  \item[(c)] one of the arcs in $F$ is $A_0$, and the other arc is
    not $A_9$. 
\end{enumerate}
The height $h(M)=(c+2n)(M)$ of $M$ is a weighted sum of the counts of crossings and
nestings in $M$, with each crossing counted once and each nesting counted twice.
Towards this weighted sum, the features satisfying condition (a) contribute
$h(M')$. By the second paragraph in the proof of Proposition \ref{prop:wt}, the features satisfying (b)
contribute exactly 
\[
c_1:=\abs{\interior(y,9)}=8-y
\] weighted counts. The weighted
count contributed by the features satisfying (c) is exactly 
the number
\[
c_2 :=  
\begin{cases}
  \abs{\interior(0,z)}-1= z-2 & \text{if $y<z$};\\
  \abs{\interior(0,z)} = z-1 & \text{if $y>z$},
\end{cases}
\]
where the extra ``$-1$" in the first case accounts for the fact that if $y<z$
then one of the crossings containing $A_0$ is $\{A_0,A_9\}$. It follows that 
\[
h(M)=h(M')+c_1+c_2.
\]
On the other hand, for the sets $C_1, C_2$ and $C_3$ defined as in Theorem
\ref{thm:mob_res} (ii), it follows by inspection that
$\abs{C_1}=c_1$, $\abs{C_2}=c_2$, and $\abs{C_3}=d_2(M')$, so we have 
\[
\rho(R(M))= c_1+c_2+d_2(M').
\]
Since $h(M')=d_2(M')$ by Proposition \ref{prop:wt},  it follows that
$\rho(R(M))=h(M)$.
\end{proof}

\begin{rmk}
\label{rmk:features}
For each $M\in \calm$, it is possible to find an explicit bijection $\zeta$ from
the crossings or nestings in $M$ (in the sense of Definition \ref{def:matching_features}) to
the crossings or nestings in $R(M)$ (in the sense of Definition \ref{def:cna}),
respectively. Therefore, we have $C(R(M))=c(M)$ and $N(R(M))=n(M)$. This allows an
alternative proof of Proposition \ref{prop:mob_level}, via the chain of
equations
\[
\rho(R(M))=C(R(M))+2N(R(M))=c(M)+2n(M)=h(M),
\]
where the first equality holds by Remark \ref{rmk:theta_C} (i).
\end{rmk}

Recall from Definition \ref{defn:S9} and Remark \ref{rmk:S9} that the Weyl group
$W$ of $\Phi$ contains a subgroup $W'\cong S_9$ generated by the set of
reflections $S'=\{s_1, s_3, s_4, \dots, s_8, s_9=s_\theta\}$ that acts on
$\Phi$ by permuting coordinates, with $s_1, s_3, s_4, \dots,$ and $s_9$ acting
as the transpositions $(12), (23), (34), \dots,$ and $(89)$, respectively,  and
that $(W',S')$ is a Coxeter system of type $A_8$. The group $S_\ten$ acts on
$\calm$ under the obvious permutation action on the elements of $\ten$, and it
is well known that $\calm$ is a quasiparabolic set for $S_\ten$ under the height
function given by $\lambda$ (see, for example, \cite[Theorem 4.6]{rains13}).
Thus, once we identify $W'$ with the natural copy of $S_9$ in $S_\ten$, the fact
that $\lambda=h$ by Proposition \ref{prop:wt} implies that $(\calm,h)$ is a
quasiparabolic set for $W'$. On the other hand, it follows from the definition
of $R(M)$ for each $M\in \calm$ that $\mob$ is closed under the action of $W'$,
so since $\ob$ is a quasiparabolic set for $W$ under the level function $\rho$,
we can naturally view $\mob$ as a quasiparabolic set for $W'$ by restriction.
The results of this section show that $\calm$ and $\mob$ are isomorphic as
quasiparabolic $W'$-sets, in the following sense.

\begin{cor}
\label{cor:matching_iso}
The map $f: \calm\ra \mob, M\mapsto R(M)$ is an isomorphism of quasiparabolic
$W'$-sets, in the sense that it is a $W'$-equivariant and level-preserving 
bijection.
\end{cor}

\begin{proof}
The map $f$ is a bijection by Lemma \ref{lemm:mob_sum}. The $W'$-equivariance of
$f$ follows readily from the definition of $R(M)$,  because each element $w\in
W'$ permutes the elements of $\ten$ pointwise and permutes the coordinate indices of
the roots in $R(M)$ in the same natural way. Finally, Proposition
\ref{prop:mob_level} proves that $f$ preserves level. 
\end{proof}

\section{Orthogonal bases via Fano labellings}
\label{sec:fob}
This section explains how we can use labellings of the Fano plane to construct
1080 orthogonal bases in $\ob$, compute their Rothe diagrams, and calculate
their levels. These 1080 orthogonal bases form a set $\fob$ that turns out to
be the complement in $\ob$ of the set $\Omega_M$ studied in Section \ref{sec:mob}. The structure
of this section parallels that of Section \ref{sec:mob}, with Section
\ref{sec:fob_def}, Section \ref{sec:fob_res}, and \ref{sec:fob_level} dealing
with the necessary definitions, Rothe diagrams, and levels, respectively. 

\subsection{Definitions}
\label{sec:fob_def}
Recall that the Fano plane is the unique projective plane $PG(2,2)$ over the
field of two elements. It contains 7 points and 7 lines such that each line
contains 3 points, each point lies on 3 lines, every pair of lines intersects in
a unique point, and every pair of points lies on a unique line. Placing a
distinct number on each point results in a labelling of the Fano plane, as
illustrated in Figure \ref{fig:fano}, and we will be interested in labellings of
the Fano plane using elements of the set $\nine$. It is well known that the
automorphism group of the Fano plane is the projective general linear group
$PGL(3,2)$, which has order $168$, so for any set $X$ of size 7, there are
$7!/168=30$ inequivalent labellings using $X$ of the Fano plane up to
automorphisms. Therefore, there is a total of $30\cdot\binom{9}{2}=1080$
inequivalent labellings of the Fano plane using elements from $\nine$.

\begin{figure}[h!]
\centering
\subfloat[$F_C$]
{
\begin{tikzpicture}[
mydot/.style={
  draw,
  circle,
  fill=black,
  inner sep=1.5pt}
]
\draw
  (0,0) coordinate (A) --
  (3,0) coordinate (B) --
  ($ (A)!.5!(B) ! {sin(60)*2} ! 90:(B) $) coordinate (C) -- cycle;
\coordinate (O) at
  (barycentric cs:A=1,B=1,C=1);
\draw (O) circle [radius=3*1.717/6];
\draw (C) -- ($ (A)!.5!(B) $) coordinate (LC); 
\draw (A) -- ($ (B)!.5!(C) $) coordinate (LA); 
\draw (B) -- ($ (C)!.5!(A) $) coordinate (LB); 
\foreach \Nodo in {A,B,C,O,LC,LA,LB}
  \node[mydot] at (\Nodo) {};    
  \node [left=0.1cm of A] {$6$};
  \node [right=0.1cm of B] {$7$};
  \node [above=0.1cm of C] {$3$};
  \node [right=0.1cm of LA] {$5$};
  \node [left=0.1cm of LB] {$4$};
  \node [below=0.1cm of LC] {$8$};
  \node [above right=0.1cm and 0.001cm of O] {$2$};
\end{tikzpicture}
}
\quad\quad\quad\quad
\subfloat[$F_N$]
{
\begin{tikzpicture}[
mydot/.style={
  draw,
  circle,
  fill=black,
  inner sep=1.5pt}
]
\draw
  (0,0) coordinate (A) --
  (3,0) coordinate (B) --
  ($ (A)!.5!(B) ! {sin(60)*2} ! 90:(B) $) coordinate (C) -- cycle;
\coordinate (O) at
  (barycentric cs:A=1,B=1,C=1);
\draw (O) circle [radius=3*1.717/6];
\draw (C) -- ($ (A)!.5!(B) $) coordinate (LC); 
\draw (A) -- ($ (B)!.5!(C) $) coordinate (LA); 
\draw (B) -- ($ (C)!.5!(A) $) coordinate (LB); 
\foreach \Nodo in {A,B,C,O,LC,LA,LB}
  \node[mydot] at (\Nodo) {};    
  \node [left=0.1cm of A] {$8$};
  \node [right=0.1cm of B] {$7$};
  \node [above=0.1cm of C] {$4$};
  \node [right=0.1cm of LA] {$9$};
  \node [left=0.1cm of LB] {$6$};
  \node [below=0.1cm of LC] {$5$};
  \node [above right=0.1cm and 0.001cm of O] {$3$};
\end{tikzpicture}}
       \caption{Two labellings of the Fano plane}
\label{fig:fano}
\end{figure}
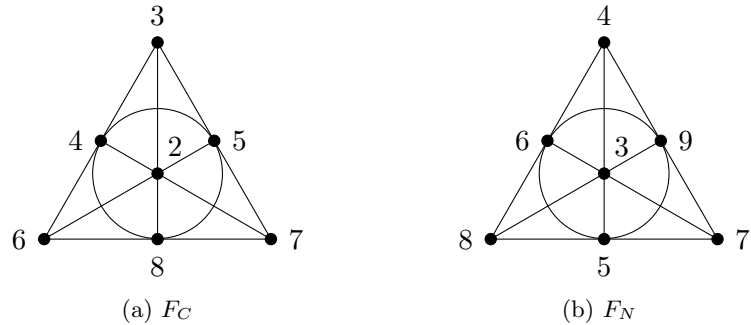

\begin{defn}\label{def:fdef}
\begin{enumerate}
  \item Let $\calf$ be the set of all labellings of the Fano plane (up to the
  automorphisms of the Fano plane) using the elements of
  $\nine$. We write each $F\in \calf$ in the form 
  \[
  F = \{t_1, t_2, \dots, t_7\}, 
  \]
  where each element $t_i$ corresponds to an edge in the Fano plane and is of
  the form $\{a_i,b_i,c_i\}$, where $a_i, b_i, c_i$ are the labels of the points
  on that edge. We  call each $t_i$ a {\it block} of $F$, and we say the three
  points $a_i, b_i, c_i$ in each block are \emph{collinear}. 

\item For each $F=\{t_i:1\le i\le 7\}\in \calf$, we call the set 
$L(F) := \bigcup_{i=1}^7 t_i$ the set of {\it labels} of $F$, and we set
$\lc(F)=\nine\setminus L(F)$.
We say $F$ is of \emph{type III} if $9\in \lc(F)$ and
  of \emph{type IV} otherwise.

\item  For any distinct labels $a,b\in L(F)$, we define $\tau(a,b)=\tau_F(a,b)$ to be the unique
  label of $F$ that is collinear with $a$ and $b$ in $F$.    
\end{enumerate}
\end{defn}

\begin{rmk}
\label{rmk:sol}
For any $F\in \calf$ and $i,j\in L(F)$, the definition of $\tau_F$ implies that $\tau(i,j)$ is the unique solution to the equation $\tau(x,i)=j$ for $x$.
\end{rmk}

The 1080 labellings in $\calf$ give rise to $1080$ elements of $\ob$ as
follows.

\begin{defn}
\label{defn:fob_construction}
For each labelling $F=\{t_i=\{a_i,b_i,c_i\}: 1\le i\le 7\}\in \calf$, we
  define 
  \[  
R(F) = \{\rr_{xy}, \ \rr_{a_1b_1c_1}, \ \rr_{a_2b_2c_2}, \ \dots, \
\rr_{a_7b_7c_7}\},
  \]
  where $\{x,y\}=\lc(F)$.
We define \[ \fob=\{R(F):F\in \calf\}. \]
\end{defn}

\begin{lemma}
\label{lemm:fob_sum}
For every labelling 
$F=\{t_i=\{a_i,b_i,c_i\}: 1\le i\le 7\}\in \calf$,
we have $R(F)\in\ob$ and
\begin{equation}
  \label{eq:fob_sum}
  \sigma(R(F)) = 4e_9 + 2e_x,
\end{equation}
where $x=\min(\lc(F))$ is the smaller of the two elements in $\lc(F)$.
The map $f: \calf\ra \fob, F\mapsto R(F)$ is bijective, so that
$\abs{\fob}=1080$.
\end{lemma}

\begin{proof}
It follows from Definition \ref{defn:fob_construction} and Remark
\ref{rmk:orthog} that $R(F)$ consists of 8 distinct and pairwise orthogonal
positive roots, so we have $R(F)\in \ob$. The surjectivity of the map $f$ then
follows from the definition of $\fob$, and $f$ is injective because the roots
$r_{a_ib_ic_i}\in R(F)$ allow us to recover $F$. It follows that
$\abs{\fob}=\abs{\calf}=1080.$

It remains to prove Equation \eqref{eq:fob_sum}. Suppose $\lc(F)=\{x,y\}$ and
$x<y$. If $F$ is of type III, then $y=9$, and direct computation shows that

\[
\sigma(R(F))=e_9-e_x-3\sum_{i\in L(F)}e_i,\] 
which is equal to $4e_9+2e_x$ by the relation $\sum_{i=1}^9e _i = 0$. If $F$ is of type
IV, then $R(F)$ consists of $p_{xy}$, three elements of the form $p_{ij9}$
corresponding to the three blocks of $F$ that contain the label $9$, and four
elements of the form $m_{ijk}$ for the blocks of $F$ that do not contain 9.
The standard forms of these elements sum to
\[
  3e_9+e_x-e_y-\sum_{i\in L(F)\setminus\{9\}}e_i,
  \]
which is equal to $4e_9+2e_x$ by the relation $\sum_{i=1}^9e _i = 0$. This
completes the proof.
\end{proof}

\begin{exa}
\label{exa:fob_exa}
The labelling 
\[
  F_C=\{\{2,3,8\},\{2,4,7\},\{2,5,6\},\{3,4,6\},\{3,5,7\}, \{4,5,8\},
\{6,7,8\}\}
\] 
shown in Figure \ref{fig:fano} (a) is of type III, and it
gives rise to the orthogonal basis 
\[
\{m_{19}, m_{238}, m_{247}, m_{256}, m_{346}, m_{357}, m_{458}, m_{678}\},
\]
which is precisely the element $\theta_C$ from Remark \ref{rmk:theta_C} (iii).
The labelling $F_N$ shown in Figure \ref{fig:fano} (b) gives rise to the orthogonal
basis
\[
\{p_{12}, p_{389}, p_{479}, p_{569}, m_{345}, m_{367}, m_{468}, m_{578}\},
\]
which is precisely the unique maximal element $\theta_N$ of $\ob$ given in Example
\ref{exa:minmax}. 
\end{exa}

\begin{cor}\label{cor:all}
The sets $\mob$ and $\fob$ partition $\ob$, and each of them is a union of
$\sigma$-classes.
\end{cor}

\begin{proof}
  Lemmas \ref{lemm:mob_sum} and \ref{lemm:fob_sum} imply that
  $\abs{\mob}+\abs{\fob}=945+1080=2025=\abs{\ob}$. The same lemmas also imply
  each element of $\mob$ has a sum of the form $4e_9+2(e_a+e_b+e_c+e_d)$ for
  some distinct elements $a,b,c,d\in \{1,2,\dots,8\}$ by Lemma
  \ref{lemm:mob_sum}, while each element of $\fob$ has a sum of the form
  $4e_9+2e_x$ for some $1\le x\le 8$ by Lemma \ref{lemm:fob_sum}. Two such sums
  cannot be equal, so $\mob$ and $\fob$ are disjoint and each of them is a union
  of $\sigma$-classes. Since $\abs{\mob}+\abs{\fob}=\abs{\ob}$, it follows that
$\mob$ and $\fob$ partition $\ob$. \end{proof}

The next definition introduces a height function on $\calf$ that is analogous to
the height function on $\calm$ from Definition \ref{def:matching_features}. As
in the case of matchings and their associated orthogonal basis, this height
function will enable us to compute the level of $R(F)$ directly in terms of $F$,
without considering roots.

\begin{defn}
  \label{def:fob_height}
Let $F\in \calf$.
\begin{enumerate}
  \item We define 
    \[
      D(F)=\{\{i,j,k\}\se L(F): \tau(i,j)<k, \tau(i,k)<j, \tau(j,k)<i\}
    \]
    and call each triple in $D(F)$ an \emph{inversion} of $F$. 

  \item We define the \emph{height} of $F$ to be the number $h(F)=\abs{D(F)}$ of
    inversions of $F$.

\item For each $z\in L(F)$, if the blocks of $F$ containing $z$ are given
  by $\{z,i_1,j_1\}, \{z,i_2,j_2\}, \{z,i_3,j_3\}$ where $i_a<j_a$ for all $a\in
  \{1,2,3\}$ and $\{i_1,i_2,i_3\}$ is a block of $F$, then we call the
  triple $\{j_1,j_2,j_3\}$ a \emph{$z$-inversion of $F$}.
\end{enumerate}
\end{defn} 

The following result shows that each $z$-inversion for a label $z\in L(F)$ is an
inversion of $F$, and that every inversion of $F$ arises this way. Thus, we may compute the inversions of
$F$ by simply checking whether each of the seven labels in $F$ induces an
inversion, rather than going through all the triples $t=\{i,j,k\}\subset L(F)$
and checking whether $t$ satisfies the conditions $\tau(i,j)<k, \tau(i,k)<j$,
and $\tau(j,k)<i$.

\begin{lemma}\label{lem:fano_inv}
Let $F \in \calf$ and let $a,b,c\in L(F)$ be three distinct labels of $F$.
\begin{enumerate}
\item{If $t$ is an inversion of $F$, then $t$ is not a block in $F$.}

\item If $t$ is not a block in $F$, then there exists a unique element
    $z\in L(F)$ such that
    \[
    L(F) = \{z, \ a, \ b, \ c, \ \tau(a, b), \ \tau(b, c), \ \tau(a, c)\}.
  \] 
  In this case, we have $\tau(a,z)=\tau(b,c), \tau(b,z)=\tau(a,c)$ and
  $\tau(c,z)=\tau(a,b)$, and the triple $\{\tau(a,b),\tau(b,c),\tau(a,c)\}$ is a
  block in $F$.

\item If $t$ is not a block in $F$ and $z$ is as in (ii), then $t$ is an
  inversion of $F$ if and only if $t$ is a $z$-inversion of $F$.

\item We have $h(F)\le 7$.
\end{enumerate}
\end{lemma}

\begin{proof}  
  If $t$ is a block, then $\tau(a,b)=c$, so $t$ cannot be an
  inversion. This proves (i). 

  Suppose that $t$ is not a block in $F$. The conclusions of (ii) can be deduced
  by repeated applications of Remark \ref{rmk:sol}, which we will use without
  further comment below. First, note that $\tau(a,b)$ cannot be $a$ or $b$, and
  we have $\tau(a,b)\neq c$ since $t$ is not a block, so
  $\tau(a,b)\notin\{a,b,c\}$. We also have $\tau(a,b)\neq \tau(a,c)$, because
  $b\neq c$. It then follows by symmetry that the set
  $Z:=\{a,b,c,\tau(a,b),\tau(a,c), \tau(b,c)\}$ has six distinct elements, which
  establishes the existence of the element $z$. 

  We cannot have $\tau(a,z)=b$ or
  $\tau(a,z)=\tau(a,b)$, because in these cases we would have $z=\tau(a,b)$ or
  $z=b$, respectively. Similarly, we cannot have $\tau(a,z)=c$ or
  $\tau(a,z)=\tau(a,c)$, so we have $\tau(a,z)=\tau(b,c)$. By symmetry, we also
  have $\tau(b,z)=\tau(a,c)$ and $\tau(c,z)=\tau(a,b)$. 

  The fact that $Z$ consists of six distinct elements implies that
  $\tau(\tau(a,b),\tau(a,c))\notin \{a,b,c\}$, and if we had $\tau(\tau(a,b),
  \tau(a,c))=z$ then we would have $\tau(\tau(a,b),z)=\tau(a,c)$, which cannot
  happen because $\tau(b,z)=\tau(a,c)$. It follows that
  $\tau(\tau(a,b),\tau(a,c))=\tau(b,c)$, which completes the proof of (ii).
 
  Suppose $t$ is not a block, and let $\{z,i_1,j_1\}, \{z,i_2,j_2\},$ and
  $\{z,i_3,j_3\}$ be the three blocks of $F$ containing $z$, with $i_a<j_a$ for
  all $a\in \{1,2,3\}$. By (ii), the three blocks of $F$ that contain $z$ are
  $\{z,a,\tau(b,c)\}, \{z,b,\tau(a,c)\}$ and $\{z,c,\tau(a,b)\}$, and the triple
  $t'=\{\tau(b,c),\tau(a,c),\tau(a,b)\}$ is a block in $F$. If $t$ is an
  inversion of $F$, then by Definition \ref{def:fob_height} (i) we have
  $\tau(a,z)=\tau(b,c)<a, \tau(b,z)=\tau(a,c)<b,$ and $\tau(c,z)=\tau(a,b)<c$.
  It follows that $\{i_1,i_2,i_3\}=t'$ and $\{j_1,j_2,j_3\}=t$, so that $t$ is a
  $z$-inversion. Conversely, if $t$ is a $z$-inversion, then by Definition
  \ref{def:fob_height} (iii) we have $t=\{j_1,j_2,j_3\}$ and
  $t'=\{i_1,i_2,i_3\}$, so it follows that $\tau(b,c)<a, \tau(a,c)<b$, and
  $\tau(a,b)<c$ and $t$ is an inversion of $F$. This completes the proof of
  (iii).

  Every inversion of $F$ is a $z$-inversion for some label $z\in L(F)$ by (i)
and (iii), and $z$ induces at most one $z$-inversion by Definition
\ref{def:fob_height} (iii), so $F$ has at most 7 inversions. This proves (iv).
\end{proof}

\subsection{Rothe diagrams}
\label{sec:fob_res}
We compute the Rothe diagram $\res(R(F))$ of each $F\in \calf$ in this section. As in
Section \ref{sec:mob_res}, we write $\al(\be)$ for $s_\al(\be)$ for all
$\al,\be\in \Phi$ throughout this section. 

\begin{theorem}
\label{thm:fob_res}
Let $F\in \calf$, and suppose that $\lc(F)=\{x,y\}$, where $x<y$. Let $R=R(F)$.
\begin{enumerate}
\item If $F$ is of type III, then we have 
$\res(R(F)) = C_1 \ \cup\ C_2 \ \cup\ C_3$, where 
\begin{align*}
C_1 &= \{\mr_{i9} : x < i < 9\},\\
C_2 &= \{\pr_{ix9} : 1\le i\le 8, i\neq x\},
\text{and}\\ 
C_3 &= \{\mr_{ijk} : i, j, k \in L(F), \ \tau(i, j)<k, \ \tau(i, k)<j\ 
\text{and\ } \tau(j, k)<i\}.
\end{align*}
\item If $F$ is of type IV, then we have $\res(R(F)) = C_1 \ \cup\ C_2
  \ \cup\ C_2'\ \cup\ C_2''\ \cup\ C_3$, where 
  \begin{align*}
C_1 &= \{\mr_{i9} : 1 \leq i \le 8\},\\
C_2 &= \{p_{xy9}\}, \\
C_2' & = \{p_{iz9}: i\in L(F)\setminus\{9\}, z\in \lc(F), \tau(i,9)>z\},\\
C_2'' &=\{
  \pr_{ij9} : 1 \leq i < j \leq 8, i,j\in L(F), \ \tau(j, 9)>i, \text{and } \tau(i,
9)>j\}, \text{and}\\ 
C_3 &= \{\mr_{ijk} : 1\le i,j,k\le 8,\ i,j,k\in L(F), \ \tau(i, j)<k, \
  \tau(i, k)<j\ 
\text{and\ } \tau(j, k)<i\}.
\end{align*}
\end{enumerate}
\end{theorem}

\begin{proof}
We can follow the same strategy as in the proof of Theorem \ref{thm:mob_res} and
compute the set $\res(R)$ using tables  \ref{tab:fob1} and \ref{tab:fob2}, where
we assume that $\al\notin R$ and analyse roots of the form $\be(\al)$ where
$\be\in \pra:=\{\be\in R: B(\al,\be)\neq 0\}$. Table \ref{tab:fob1} deals with
the labellings of type III, in which case we have $y=9$ and $R=\{m_{x9}\}\cup
\{m_{a_ib_ic_i}:1\le i\le 7\}$, where the triples $\{a_i,b_i,c_i\}$ are the
blocks of $F$. Table \ref{tab:fob2} (shown in Appendix \ref{sec:wide_tables})
treats type IV, in which case $R$ consists of $p_{xy}$, three elements of the
form $p_{ij9}$ corresponding to the blocks of $F$ containing 9, and four
elements of the form $m_{ijk}$ corresponding to the remaining blocks of $F$. 
The tables should be read with the next three paragraphs in mind.

\begin{table}[h!]
  \centering
\begin{tabular}{@{}lll@{}}
\toprule
  $\al$ & $\be\in \pra$ & $\be(\al)$ \\
  \midrule
  $p_{ij}, j=x$ & $m_{x9}$  & $p_{i9}<0$ \\
  $p_{ij}, j\neq x$ & $m_{i'jk'}$, where $\{i',j,k'\}$ is any
  $i$-excluding block & $p_{ii'k'}<0$ \\
  \midrule
  $m_{i9}$ (so $i\in L(F)$) & $m_{x9}$, or $m_{ijk}$ for the blocks of the form
  $\{i,j,k\}$  & $m_{ix}=p_{xi}$, or $p_{jk9}>0$ \\
  \midrule
  $p_{ij9}, x\notin \{i,j\}$ & $m_{x9}$ & $p_{ijx}<0$\\
  $p_{ix9}, i\notin \{x,9\}$ & $m_{i'jk}$ for the $i$-excluding blocks
  $\{i',j,k\}$ & $m_{i'jk}(p_{ix9})>0$\\
  \midrule
  $m_{ijk}, x\in \{i,j,k\}$ & $m_{x9}$ & $m_{ab9}<0$, where
  $\{a,b,x\}=\{i,j,k\}$\\
  $m_{ijk}, x\notin \{i,j,k\}$ & $\gamma:=m_{\tau(i,j)\tau(i,k)\tau(j,k)}, m_{ij\tau(i,j)},m_{ik\tau(i,k)}, m_{jk\tau(j,k)}
  $ &
  $\gamma(m_{ijk})>0$, $p_{\tau(i,j)k},p_{\tau(i,k)j}, p_{\tau(j,k)i}$\\
  \bottomrule
\end{tabular}
\caption{Analysis of $\be(\al)$, III}
\label{tab:fob1}
\end{table}

In the seventh row of Table \ref{tab:fob2}, the element $z'$ stands for the
unique element such that $\{z,z'\}=\{x,y\}$. The last three elements listed in
the second column of this row correspond to the three blocks containing
$\tau(i,9)$, with the two $i$-excluding blocks $\{i',j',\tau(i,9)\}$ being also
necessarily $9$-excluding and thus disjoint from $\{i,z,9\}$. 

In the last block of Table \ref{tab:fob2}, i.e., the block treating the positive
roots $\al\notin R$ of the form $m_{ijk}$ where $i,j,k\in \{1,2,\dots, 8\}$, case (1) refers to the situation
where we can find a block of the form $T=\{i',j',9\}$ in $F$ that is disjoint
from $\{i,j,k\}$. This may occur in one of the following two ways.
\begin{enumerate}
  \item We have $\{i,j,k\}\cap\{x,y\}\neq \emptyset$, so that the set
    $Z:=\{i,j,k\}\setminus\{x,y\}=\{i,j,k\}\cap L(F)$ contains at most two
    elements, and $F$
    contains at
    most two blocks of the form $\{z,9,\tau(z,9)\}$ where $z\in Z$. Any other
    block in $F$ that contains $9$ suffices as $T$ in this subcase.
  \item We have $\{i,j,k\}\cap \{x,y\}=\emptyset$, and $9$ is collinear with two
    elements in the set $i,j,k$. We may assume without loss of generality that
    $\{i,j,9\}$ is a block in this subcase, so that one of the other two blocks
    avoids $k$ and suffices as $T$.
\end{enumerate}
Case (2) in the same block represents the only remaining possibility, which is
the possibility that 
$\{i,j,k\}\cap \{x,y\}=\emptyset$ and no pair of elements from $\{i,j,k\}$ is
collinear with $9$. In other words, we have $9\notin
\{\tau(i,j),\tau(i,k),\tau(j,k)\}$. In this case, it follows from Lemma \ref{lem:fano_inv}
(ii) that $\{\tau(i,j),\tau(i,k),\tau(j,k)\}$ is a block in $F$, so that the
element $\gamma:=m_{\tau(i,j)\tau(i,k)\tau(j,k)}$ lies in $\pra$. 

For the roots $\be$ listed in the second columns of the tables, the fact that
$\be\in \pra$ follows from direct inspection and Remark \ref{rmk:orthog} as in
the case of tables \ref{tab:mob1} and \ref{tab:mob2} in most cases, with the only exceptions being the three
roots denoted by $\gamma$. These roots appear in the last row of table
\ref{tab:fob1}, third to last row of Table \ref{tab:fob2}, and the last row of
Table \ref{tab:fob2}, and the fact that $\gamma\in \pra$ in these cases follows
from Lemma \ref{lem:res_prep} (ii), as in the last sentence of the previous
paragraph. The assertions concerning $\be(\al)$ indicated in the third columns
of tables \ref{tab:fob1} and \ref{tab:fob2} can be readily checked using Lemma
\ref{lem:res_prep} as before, and these assertions immediately imply the desired
conclusions of the theorem. \end{proof}

\subsection{Levels}
\label{sec:fob_level}
Theorem \ref{thm:fob_res} allows us to compute the level of $R(F)$ via the height
of $F$ for each $F\in \calf$, as follows.

\begin{prop}\label{prop:fob_level}
 We have
\[
\rho(R(F)) =h(F)+ 24 - \sum_{z\in\lc(F)} z \]
for all $F\in \calf$.
\end{prop}

\begin{proof}
  If $F$ is of type III, then $\lc(F)=\{x,9\}$ for some $1\le x\le 8$, so that
  for the sets $C_1, C_2,$ and $C_3$ from Theorem \ref{thm:fob_res} (i), we have
  $\abs{C_1}=8-x$
  and $\abs{C_2}=7$. 
  The set $C_3$ is indexed by precisely the inversions of $F$, so
  we have
  $\abs{C_3}=h(F)$ 
  and 
  \[
    \rho(R(F))=\abs{C_1}+\abs{C_2}+\abs{C_3}=15-x+h(F)=h(F)+24-\sum_{z\in\lc(F)}z.
  \]

  Now suppose that $F$ is of type IV and that $\lc(F)=\{x,y\}$, where $x<y<9$.
  Let $C_1, C_2, C_2', C_2'',$ and $C_3$ be as in Theorem \ref{thm:fob_res}
  (ii). It is evident that $\abs{C_1}=8$ and $\abs{C_2}=1$. The set $C_2'$ can
  be written as the disjoint union of the sets 
  \[X:=\{p_{ix9}: i\in L(F)\setminus\{9\} \text{
    and } 
    \tau(i,9)>x\} \text{ and } Y:=\{p_{iy9}: i\in L(F)\setminus\{9\} \text{ and }
\tau(i,9)>y\}.\] The set
  $Y$ has $8-y$ elements, namely, the distinct solutions for $i$ to the equations
  of the form $\tau(i,9)=y'$ where $y'$ is a label larger than
  $y$. Similarly, since there are $7-x$ labels in $L(F)$ larger than $x$, we have
  $\abs{X}=7-x$. 

  If $\{9,a,b\}, \{9,c,d\}, \{9,e,f\}$ are the three blocks
  containing $9$ in $F$, then the partition $M=\{\{a,b\}, \{c,d\}, \{e,f\}\}$
  may be viewed as a perfect matching of the set $L(F)\setminus\{9\}$ with the
  property that $\tau_M(i)=\tau(i,9)$, and then
  the set $C_2''$ is precisely the involution Rothe diagram $D_{\texttt{FPF}}(M)$ as defined
  in Definition \ref{def:matching_features} (v). The equality $d_1(M)=d_2(M)$
  from Proposition \ref{prop:wt} then implies that $\abs{C''_2}=\abs{Z}$ for the
  set
  \[
    Z:=\{p_{ij9}: i,j\in L(F), \tau(i,9)<j, \tau(j,9)<i\}
  \]
  and $Z$ in turn has the same size as the set
\[
  Z':= \{p_{ij9}: i,j\in L(F), \tau(i,9)<j, \tau(j,9)<i, \tau(i,j)<9\},
  \]
  where the last equality holds because $9$ is the largest label and is not
  collinear with $i$ and $j$ whenever $\tau(i,9)<j$. The triples indexing the elements of $Z'$ and $C_3$
  partition the inversions of $F$, so we have 
  \[
    \abs{C_2''}+\abs{C_3}=\abs{Z'}+\abs{C_3}=h(F).
  \]

  Combining the results of last two paragraphs yields
  \[
    \rho(R(F))=8+1+(8-y)+(7-x)+h(F)= h(F)+24-\sum_{z\in \lc(F)} z,
  \]
  which completes the proof. 
\end{proof}

\section{Enumerative results}
\label{sec:enum}
In this section, we compute the level generating functions of the
$\sigma$-classes in $\ob$ and of $\ob$ itself. Recall that the $\sigma$-classes
are indexed by the sequences in the set $\cals$ via the poset isomorphism $\Psi$
from Definition \ref{defn:iso_chain}, and that Theorem \ref{thm:gen_func}
expresses the generating functions for the $\sigma$-classes in terms of these
sequences. We will analyse how the isomorphism $\Psi$ is highly compatible with
the partition of $\ob$ into $\mob$ and $\fob$ in Section \ref{sec:seq_to_sum}.
This analysis will then allow us to apply the results of sections \ref{sec:mob}
and \ref{sec:fob} to obtain the level generating functions of the
$\sigma$-classes in Section \ref{sec:gen_func}. We will also use these
generating functions to deduce the Poincar\'e polynomial of $\ob$.

\subsection{Sums via sequences}
\label{sec:seq_to_sum}
Maintain the notation of Section \ref{sec:sums}, and recall that we 
have a chain of
natural poset isomorphisms
\begin{equation*}
  \Psi:
 \cals\stackrel{\psi}{\ra} J(P) \stackrel{\varphi}{\ra} \calj
  \stackrel{\phi}{\ra} \nn(\ob)\stackrel{\sigma}{\ra} \Sigma
\end{equation*}
that allows us to index the sums in the distributive lattice
$\Sigma=\{\sigma(R):R\in \ob\}$ using the
sequences in the set 
\[
\cals=\{(d_1,d_2,d_3,d_4): 0 \le d_i\le c_i \text{ for all $i\in [4]$ and
$d_{i-1}\ge \min(d_i-1,c_{i-1})$ for all $i\in \{2,3,4\}$}\}.
\]
Recall that $\abs{\Sigma}=50$, that $\calj$ is the set $\{\nu\in W:\nu\le_L
\wnn\inverse\}$ where $\wnn$ is
the fully commutative element given by 
\[
    \wnn=(s_1s_4s_6s_8)(s_3s_5s_7)(s_4s_6)(s_2s_5)(s_4)(s_3)(s_1), 
\]
that $P$ is the heap poset of $\wnn$ as shown in Figure \ref{fig:heap}, and that
$J(P)$ is the set of ideals of $P$. 

Recall each $\sigma$-class of $\ob$ lies entirely in either $\mob$ or $\fob$ by
Corollary \ref{cor:all}. One goal of this subsection is to explain how the
partition of $\ob$ into $\mob$ and $\fob$ is reflected in the sets $\cals, J(P),
\calj, \nn(\ob),$ and $\Sigma$, so that we will be able to compute $\Psi(\seqd)$
directly in terms of $\seqd$ for each $\seqd\in \cals$ (Proposition
\ref{prop:seq_to_sum}). Another goal of this subsection is to study certain
nonnesting perfect matchings of $\ten$ associated to the sequences $\seqd\in \cals$
that correspond to the $\sigma$-classes in $\mob$ (Lemma \ref{lem:lr}), to obtain results that are necessary for the proofs of Section
\ref{sec:gen_func}. 

Let $\seqd_M=(1,2,3,4)\in \cals$ and let $J_M=\psi(d_M)$, so that $J_M$ consists
of all elements in $P$ under the dashed line in Figure \ref{fig:heap}, as
we observed in Example \ref{exa:bijections}. 
Set $w_M=\varphi(J_M)$.
Let 
\[
  \cals_M=\{\seqd\in \cals: \seqd\le \seqd_M\}=\{(d_1,d_2,d_3,d_4)\in \cals: 0\le d_i\le i
  \text{ for all } i\},
\]
where $\le$ stands for the entrywise comparison order on $\cals$, and let
$\cals_F=\cals\setminus\cals_M$. Finally, let 
 \[ \cali_M= \{K\in J(P): K\se J_M\},\] 
and  
\[
  \cali_F=J(P)\setminus\cali_M=\{K\in J(P): K\not\se J_M\}. 
\] 
There is a unique element labelled by the generator $s_2$ in the heap $P$. If we
denote this element by $p_0$, then $p_0$ is the unique minimal element of
$P\setminus J_M$ by inspection of Figure \ref{fig:heap}, so $\cali_F$ and
$\cali_M$ can be characterized as the sets of all ideals in $P$ that do and do
not contain $p_0$, respectively. The sets $\cali_F$ and $\cali_M$ also have
convenient characterizations in terms of their corresponding sequences in $\cals$:
since $J_M=\psi(\seqd_M)$, we have 
\begin{equation}
\label{eq:cali_M}
\cali_M=\{\psi(\seqd):\seqd\in \cals_M\}=\{\psi((d_1,d_2,d_3,d_4)): 0\le d_i\le i
  \text{ for all } i\};
\end{equation}
and since an ideal of $P$ contains $p_0$ if and only if
it contains all points in $C_1, C_2$, and $C_3$, we have 
\begin{equation}
\label{eq:cali_F}
\cals_F=\{(1,2,4,j):0\le j\le 7\} \text{\; and \;}
  \cali_F=\{\psi((1,2,4,j)): 0\le j\le 7\}.
\end{equation}
In particular, we have $\abs{\cali_F}=8$, and the fact that
$\abs{J(P)}=\abs{\Sigma}=50$ then implies that $\abs{\cali_M}=42$.  

\begin{rmk}
\label{rmk:42}
\begin{enumerate}
\item 
The fact that $\abs{\cali_M}=42$ can be proved in many other ways, for example,
by noting that integer sequences of the form $(d_1,d_2,d_3,d_4)$ where $0\le
d_i\le i$ are enumerated by the Catalan number 42. Another proof relevant to us
proceeds as follows: when viewed as a poset itself, the ideal $J_M$, whose Hasse
diagram is shown in Figure \ref{fig:a4}, is isomorphic to the root poset of type
$A_4$; the ideals of this root poset are well known to be in a natural bijection
with the 42 Dyck paths of semilength 5 \cite[Section 5.1.3]{ringel16}, as well
as with the 42 nonnesting perfect matchings of any set of size 10
\cite[Interpretation 64]{stanley15}, so it follows that we have
$\abs{\cali_M}=42$.  \item The number $\abs{\cali_F}=8$ is also the number of
possible sums for the elements of $\fob$ by Lemma \ref{lemm:fob_sum}, because we
have $\{\min(\lc(F)):F\in \calf\}=\{1,2,\dots, 8\}$. \end{enumerate} \end{rmk}

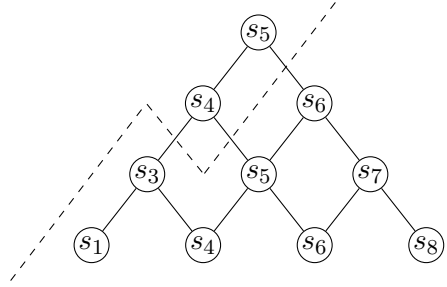
\begin{figure}[h!]
\begin{tikzpicture}
\node[main node] (0) {$s_1$};
\node[main node] (1) [right=1cm of 0] {$s_4$};
\node[main node] (2) [right=1cm of 1] {$s_6$};
\node[main node] (3) [right=1cm of 2] {$s_8$};
\node[main node] (4) [above right=0.6cm and 0.4cm of 0] {$s_3$};
\node[main node] (5) [above right=0.6cm and 0.4cm of 1] {$s_5$};
\node[main node] (6) [above right=0.6cm and 0.4cm of 2] {$s_7$};
\node[main node] (7) [above right=0.6cm and 0.4cm of 4] {$s_4$};
\node[main node] (8) [above right=0.6cm and 0.4cm of 5] {$s_6$};
\node[main node] (9) [above right=0.6cm and 0.4cm of 7] {$s_5$};

\node[small node] (x) [below left = 0.3cm and 0.9cm of 0] {};
\node[small node] (y) [left = 0.5cm of 7] {};
\node[small node] (z) [right=0.5cm of 4] {};
\node[small node] (w) [above right = 0.3cm and 0.9cm of 9] {};

\path[draw]
(0)--(4)--(1)--(5)--(2)--(6)--(3)
(4)--(7)--(5)--(8)--(6)
(8)--(9)--(7);
 \draw[dashed] 
  (x)--(y)--(z)--(w);
\end{tikzpicture}
    \caption{The ideal $J_M$}
    \label{fig:a4}
\end{figure}

By Remark \ref{rmk:42}, the numbers $\abs{\cali_M}=42$ and $\abs{\cali_F}=8$ are
also the number of nonnesting perfect matchings of the set $\ten$ and the number
of $\sigma$-classes in $\fob$, respectively. This is not a coincidence, but a
reflection of the fact that the bijection $\phi\circ\varphi: J(P)\ra \nn(\ob)$
naturally restricts to two bijections, one from $\cali_M$ to the set
$\nn(\ob)\cap \mob$, whose elements correspond bijectively to the nonnesting
perfect matchings of $\ten$ by Remark \ref{rmk:features}, and the other from
$\cali_F$ to the set $\nn(\ob)\cap \fob$, whose elements bijectively represent
the $\sigma$-classes partitioning $\fob$ by Proposition \ref{prop:Sigma}. This
fact, which will be proved soon in Proposition \ref{prop:seq_to_sum}, is a
crucial connection between the results of Section \ref{sec:sums} and the theory
of sections \ref{sec:mob} and \ref{sec:fob}. 

\begin{defn}
\label{defn:seq_to_word}
Maintain all notation recalled or introduced in this subsection, and let $C'_1,
C'_2, C'_3$ and $C'_4$ be the four Northeast--Southwest diagonal
  chains appearing from right to left in Figure \ref{fig:a4}.
\begin{enumerate}
\item 
Let 
\[
  w_1=s_1, w_2=s_3s_4,
w_3=s_2s_4s_5s_6, \text{ and }w_4=s_1s_3s_4s_5s_6s_7s_8,\]
so that $\wnn\inverse=w_4w_3w_2w_1$.
For each sequence $\seqd=(d_1,d_2,d_3,d_4)\in \cals$,
let $\nu_i$ be the element
expressed by the suffix of $w_i$ consisting of the last $d_i$ elements for each
$i$, and
define
\[
  \nu_{\seqd}=\nu_4\nu_3\nu_2\nu_1.
\]
\item Let 
  \[
  w'_1=s_8, w'_2=s_7s_6, w'_3=s_6s_5s_4, \text{ and } w'_4=s_5s_4s_3s_1,
  \]
  so that $w_M\inverse = w'_4w'_3w'_2w'_1$.  
For each integer sequence $\seqd'=(d'_1,d'_2,d'_3,d'_4)$ where $0\le d'_i\le i$ for all
$i$, let $\nu'_i$ be the 
element expressed by the suffix of $w'_i$ consisting of the last $d'_i$ elements for
each $i$, and
define
\[
  \nu'_{\seqd'}=\nu'_4\nu'_3\nu'_2\nu'_1.
\]
\item For each ideal $K\in J(P)$, we define $\seqd_K=(d_1,d_2,d_3,d_4)$ where
  $d_i=\abs{K\cap C_i}$ for each $i$. For each ideal $K\in \cali_M$, we define
  $\seqd'_K=(d'_1,d'_2,d'_3, d'_4)$ where $d'_i=\abs{K\cap C'_i}$ for each $i$,
  and we define the \emph{mirror} of $K$ to be the ideal $\hat K=\psi(\seqd'_K)$.
\end{enumerate} \end{defn}

Note that for each $K\in J(P)$, the sequence $\seqd_K$ is precisely
$\psi\inverse(K)$ by the definition of $\psi$. The symmetry of Figure
\ref{fig:a4} guarantees that $\hat K$ is indeed just the mirror image of $K$
across the central vertical axis of the figure (i.e., the line going through two
elements labelled by $s_5$), and that we have $\seqd_{\hat K}=\seqd'_K$. 

\begin{exa}
\label{exa:dd'}
For the ideal $K\in \cali_M$ consisting of the elements below the dashed lines in
Figure \ref{fig:a4}, we have $\seqd_K=(1,2,2,3)$ and $\seqd'_K=(1,2,3,2)$, and we have
\[
  \nu_{\seqd_K}=(s_6s_7s_8)(s_5s_6)(s_3s_4)(s_1)\text{\quad and\quad}
  \nu_{\seqd'_K}=(s_3s_1)(s_6s_5s_4)(s_7s_6)(s_8).
\]
The mirror
of $K$ is $\hat K=\psi((1,2,3,2))$, which can be obtained from $K$ by removing
its unique maximal element labelled by $s_6$ and then adding the element
labelled by $s_4$ above the dashed lines.
\end{exa}

\begin{lemma}
\label{lem:d_d'}
Let $\seqd\in \cals$, and let $w_{\seqd}=(\varphi\circ\psi)(\seqd)$ be the corresponding element
in $\calj$.
\begin{enumerate}
  \item If $\seqd\in \cals$, then $\nu_{\seqd}$ is a reduced word for $w_{\seqd}$.
\item Let $\seqd\in \cals_M$. If $K=\psi(\seqd)$ is the corresponding ideal in
  $\cali_M$ and $\seqd'=\seqd'_K$, then we have   
\begin{equation}
\label{eq:lin_ext}
w_{\seqd}=\nu_{\seqd}=\nu'_{\seqd'},
\end{equation}
where $\nu_{\seqd}$ and $\nu'_{\seqd'}$ are reduced words of $w_{\seqd}$,  
  and the multiset of numbers appearing in $\seqd$ equals the
  multiset of numbers appearing in $\seqd'$.  
\end{enumerate}
\end{lemma}

\begin{proof}
Let $\seqd\in \cals$ and let $K=\psi(\seqd)$ be the corresponding ideal. Inspection of
Figure \ref{fig:heap} shows that listing the elements of $I_{\seqd}\cap C_i$ from
bottom to top, in increasing order for $i$, results in a linear extension of
$I_{\seqd}$. The labels of these elements in the heap $P$ are precisely the letters
appearing in the word $\nu_{\seqd}$ from right to left, so it follows from the
definition of $\varphi$ that $\nu_{\seqd}$ is a reduced word for $w_d$. This proves (i).

Now suppose that $\seqd\in \cals_M$. The fact that $\nu_{\seqd}$ is a reduced word for
$w_{\seqd}$ follows from (i) by restriction, and similarly $\nu'_{\seqd'}$ is a reduced word
for $w_{\seqd}$ because the word $\nu'_{\seqd'}$ corresponds to another linear extension
of $I_{\seqd}$, namely, the one obtained by ordering the elements as they appear in
the chains $C'_1, C'_2, C'_3,$ and $C'_4$.

It remains to prove that $\seqd$ and $\seqd'$ contain the same multiset of numbers when
$\seqd\in \cals_M$. This follows from the chain of equalities
\[
  \mset{\seqd}=\mset{\seqd_K}=\mset{\seqd_{\hat K}}=\mset{\seqd'_{K}}=\mset{\seqd'}, 
\]
where $\mset{e}$ denotes the multiset of elements in $e$ for each sequence $e$.
Here, the first equality holds because $\seqd=\psi\inverse(K)=\seqd_K$; the
second equality holds because the fact that $K$ and $\hat K$ are left-right
mirror images of each other implies that, as ordered sequences, $\seqd_K$ and
$\seqd_{\hat K}$ are reversals of each other; the third equality holds because
$\seqd_{\hat K}=\seqd'_{K}$ as we observed earlier; and the last equality holds
because $\seqd'=\seqd'_K$ by definition. The proof is now complete. \end{proof}

The following lemma will provide one method for computing $\Psi(\seqd)$ for any
$\seqd\in \cals_M$. Recall from sections \ref{sec:Phi} and \ref{sec:mob_level} that
the simple reflections $s_1, s_3, s_4, \dots, s_8, s_9=s_\theta$ in the Weyl
group $W=W(E_8)$ generate a copy $W'$ of the symmetric group $S_9$ inside $W$,
and that $W'$ naturally acts on the set $\calm$ of perfect matchings of $\ten$,
with $s_1, s_3, s_4, \dots, s_7$, and $s_8$ acting as the transpositions
$(12), (23), (34),\dots, (67)$, and $(78)$, respectively.  In the proof of the
 lemma, we will consider each arc $(i,j)$ in the perfect matching
$M_A=\{(0,1),(2,3), (4,5), (6,7), (8,9)\}$ as an elastic string with two
objects, $o_i$ and $o_j$, attached at the left and right endpoints.
We will consider the numbers $0,1,2,\dots, 8,9$ as \emph{positions}, with $o_i$
occupying position $i$ for each $i$, and we will think of the actions of the
permutations in $W'$ on $\calm$ as rewirings of arcs that move the objects $o_i$
to different positions. For example, if we apply the simple reflection
$s_4=(3,4)$ to $M_A$, then we would cross the arcs $(2,3)$ and $(4,5)$, moving
$o_3$ to position 4 and $o_4$ to position 3.

\begin{lemma}
\label{lem:lr}
Let $M_A=\{(0,1),(2,3), (4,5), (6,7), (8,9)\}\in
\calm$. Let $\seqd=(d_1,d_2,d_3,d_4)\in \cals_M$, $K=\psi(\seqd)\in \cali_M$,
$w_{\seqd}=(\varphi\circ\psi)(\seqd)\in \calj$, and $R_{\seqd}=w_{\seqd}(\theta_A)\in \nn(\ob)$. If
$\seqd\in\cals_M$ and $\seqd'_K=(d'_1,d'_2,d'_3,d'_4)$, then $w_{\seqd}(M_A)$ is the
nonnesting perfect matching in $\calm$ whose left endpoints are
$0,2-d_1,4-d_2,6-d_3,$ and $8-d_4$ and whose right endpoints are $d'_1+1,
d'_2+3, d'_3+5, d'_4+7$, and $9$, where both lists are given in increasing
order. 
\end{lemma}

\begin{proof}
  Let $M=w_{\seqd}(M_A)$. Define $\nu_{\seqd}=\nu_4\nu_3\nu_2\nu_1$ and
  $\nu'_{\seqd'}=\nu'_4\nu'_3\nu'_2\nu'_1$ as in Definition
  \ref{defn:seq_to_word} (i) and (ii), respectively, so that
  $w_{\seqd}=\nu_{\seqd}=\nu_{\seqd'}$ by Lemma \ref{lem:d_d'} (ii). It follows
  that $M$ can be obtained from $M_A$ by successive applications of $\nu_1,
  \nu_2, \nu_3,$ and then $\nu_4$, or of $\nu'_1, \nu'_2, \nu'_3$, and then
  $\nu'_4$.

Let $i\in \{1,2,3,4\}$, and let $M_{\le i}=\nu_{i-1}\cdots\nu_2\nu_1(M_A)$. The
definition of $\nu_i$ implies that as a permutation, $\nu_i$ is the $(d_i+1)$-cycle given by 
\[
  \nu_i=(2i,
2i-d_i, 2i-d_i+1, 2i-d_i+2, \dots, 2i-1).
\]
It follows that applying $\nu_i$ to $M_{\le i}$ fixes all
objects $o_j$ where $j>2i$, shifts $2i$ to the left by $d_i$ positions, and
shifts every object of the form $o_j$ where $2i-d_i\le j<2i$ to the right by one
position. The definition of $\cals$ guarantees that $d_{i-1}\ge d_i-1$, so we
have 
\[
  (2i-2)-d_{i-1}< 2i-d_i.
\]
Thus, as we apply $\nu_1, \nu_2,\nu_3,$ and $\nu_4$ to $M_A$
successively, the objects $o_{0}, o_{2}, o_4, o_6,$ and $o_8$ remain attached to
the left endpoints of the five arcs throughout the process, and the positions
they occupy remain in the same relative order, with each $o_i$ being fixed by
every $\nu_j$ where $j<i$, moved to position $2i-d_i$ by $\nu_i$, and then fixed
at position $2i-d_i$ by every $\nu_j$ where $j>i$. In particular, when $\nu_i$
shifts $o_{2i}$ to the left by $d_i$ positions and shifts the objects $o_j$
where $2i-d_{i}\le j<2i$ to the right by one position, it only moves $o_{2i}$
past right endpoints of arcs. Doing so only creates crossings and no nestings. 

It follows from the previous paragraph that $M$ is a nonnesting perfect matching
with left endpoints $0,2-d_1, 4-d_2, 6-d_3,$ and $8-d_4$, where the endpoints
are listed in increasing order. A similar argument using the fact that
$w_{\seqd}=\nu'_{\seqd'}$ shows that the right endpoints of $M$ are $d_1'+1,
d_2'+3,d'_3+5, d'_4+7$, and $9$, where the elements are listed in increasing
order. This completes the proof.
\end{proof}
%
%
%
%
%
\begin{prop}
\label{prop:seq_to_sum}
Let $\seqd=(d_1,d_2,d_3,d_4)\in \cals$.
\begin{enumerate}
\item If $\seqd\in \cals_M$, then we have
\[
  \Psi(\seqd)=4e_9+2\sum_{i=1}^4 e_{2i-d_i}.
\]
\item If $\seqd\in \cals_F$, so that $\seqd=(1,2,4,j)$ for some $0\le j\le 7$, then we have 
\[
  \Psi(\seqd)=4e_9+2e_{8-j}.
\]
\end{enumerate} 
\end{prop}

\begin{proof}
  Let $w_{\seqd}=(\varphi\circ \psi)(\seqd)$, so that
  $\Psi(\seqd)=\sigma(w_{\seqd}(\theta_A))$. Let $\theta_A, \theta_C, M_A,$ and
  $F_C$ be as defined in Example \ref{exa:minmax}, Remark \ref{rmk:theta_C}
  (iii), Example \ref{exa:mob_exa}, and Example \ref{exa:fob_exa}, respectively.
  Recall that $R(M_A)=\theta_A$ by Example \ref{exa:mob_exa} and
  $R(F_C)=\theta_C$ by  Example \ref{exa:fob_exa}. 
  If $\seqd\in \cals_M$, then since the map $f:\calm\ra \mob, M\mapsto R(M)$ is
  $W'$-equivariant by Corollary \ref{cor:matching_iso}, we have
  $w_{\seqd}(\theta_A)=w_{\seqd}(R(M_A))=R(w_{\seqd}(M_A))$. The element
  $w_{\seqd}(M_A)$ is a perfect
  matching with left endpoints $0, 2-d_1, 4-d_2, 6-d_3$, and $8-d_4$, so it
  follows from Lemma \ref{lemm:mob_sum} that \[
    \Psi(\seqd)=\sigma(w_{\seqd}(\theta_A))=\sigma(R(w_{\seqd}(M_A)))= 4e_9+2\sum_{i=1}^4
e_{2i-d_i}. \] This proves (i). 

Part (ii) can be proved similarly. The key is to note that if
$\seqd=(1,2,4,j)$ and $\wnn$ is the fully commutative element with the property that
$\wnn(\theta_C)=\theta_A$ from Proposition \ref{prop:Sigma}, then the definitions of
$\varphi$ and the heap $P$ imply that we have $w_{\seqd}=\nu'\wnn\inverse$ where $\nu'$
is the group element expressed by the suffix of the word $u={s_8s_7s_6s_5s_4s_3s_1}$
consisting of the last $(7-j)$ letters, which are the labels of the top $(7-j)$
elements in the chain $C_4$ in the heap $P$. The simple reflections $s_1, s_3,
\dots, s_7$, and $s_8$ act as the transpositions $(12), (23), \dots, (67)$, and
$(78)$ in $W'$ on the elements of $\calf$, so the action of $\nu'$ 
on $F_C$ results in a labelling $F$ of the Fano plane for which
$\lc(F)=\{8-j,9\}$. The map $f:\calf\ra \fob, F\mapsto R(F)$ is $W'$-equivariant
by Definition \ref{defn:fob_construction}, so we have 
\[
  w_{\seqd}(\theta_A)=(\nu'\wnn\inverse)(\theta_A)=\nu'(\theta_C)=\nu'(R(F_C))=R(\nu'(F_C))=R(F),
\]
and it then follows from Lemma \ref{lemm:fob_sum} that 
\[
  \Psi(\seqd)=\sigma(w_{\seqd}(\theta_A))=\sigma(R(F))=4e_9+2e_{8-j},
\]
which proves (ii).
\end{proof}

\subsection{Generating functions}
\label{sec:gen_func}

Maintain the notation of Section \ref{sec:seq_to_sum}.  We are now ready to
prove Theorem \ref{thm:gen_func}. Recall that the theorem asserts that for any
sequence  $\seqd=(d_1,d_2,d_3,d_4)\in \cals$ and the element $\gamma=\Psi(\seqd)\in
\Sigma$, the level generating function for the $\sigma$-class in $\ob$
corresponding to $d$ can be expressed as
\begin{equation} \label{eq:seq_to_gen} \sum_{R\in
\sigma\inverse(\gamma)}q^{\rho(R)}=q^{\abs{\seqd}}\prod_{i=1}^4 [d_i+1]_q,
\end{equation}
where $\abs{\seqd}=d_1+d_2+d_3+d_4$.

\begin{proof}[Proof of Theorem \ref{thm:gen_func}]
  Let $\seqd=(d_1,d_2,d_3,d_4)\in \cals$ and $\gamma=\Psi(\seqd)$.
  By Proposition \ref{prop:seq_to_sum}, we have
  $\gamma=4e_9+2\sum_{i=1}^4e_{2i-d_i}$ if $\seqd\in \cals_M$
  and $\gamma=4e_9+2e_{8-d_4}$ if $\seqd\in \cals_F$.  
   Let 
  \[
       \calm_\gamma= \{M\in \calm: \text{the left endpoints of 
    $M$ are $0, 2-d_1, 4-d_2, 6-d_3,$ and $8-d_4$}\}
  \]
 and  
  \[
    \calf_\gamma=\{F\in \calf: \min(\lc(F))=8-d_4\}=\{F\in \calf: \lc(F)=\{8-d_4,y\},
    \text{where $8-d_4<y\le 9$}\}.
  \]
 By Lemma \ref{lemm:mob_sum}, Lemma \ref{lemm:fob_sum}, Corollary \ref{cor:all},
  and Proposition \ref{prop:seq_to_sum}, we have 
 \[
\sigma\inverse(\gamma)=\{R(M): M\in \calm_\gamma\}\se \mob\]
if $d\in \cals_M$
and 
\[
    \sigma\inverse(\gamma)=\{R(F): F\in \calf_\gamma\}\se \fob
 \] 
 if $d\in \cals_F$.

Suppose that $\seqd\in \cals_M$. Let $K=\psi(\seqd)$ and $\seqd'_K=(d'_1,d'_2,d'_3,d'_4)$ be
as in Lemma \ref{lem:lr}, so that $\calm_\gamma$ consists of the elements of
$\calm$ with right endpoints $r_1:=d'_1+1, r_2:=d'_2+3, r_3:=d'_3+5,
r_4:=d'_4+7$, and $r_5:=9$ by Lemma \ref{lem:lr}. In \cite[Section
3.3]{watson14}, Watson defines an equivalence relation denoted $\Delta$ on the
set of all perfect matchings of the set $[2n]=\{1,2,\dots, 2n\}$ for any
positive integer $n$, where each equivalence class for $\Delta$ consists of all
perfect matchings of the set $[2n]$ with a fixed set of right endpoints. (Watson
denotes such right endpoints as ``$d_i(\delta)$", where $\delta$ refers to
certain Dyck paths associated with perfect matching and the symbol ``$d_i$" is
not to be confused with our notation for the entries of the sequences $\seqd\in
\cals$.) It follows from the proof of \cite[Theorem 1]{watson14} that 
    \[
      \sum_{M\in \calm_\gamma} q^{\wt(M)}=q^{\sum_{i=1}^4
      (r_i-(2i-1))}\prod_{i=1}^4 [r_i-(2i-1)+1]_q=q^{d'_1+d'_2+d'_3+d'_4}\prod_{i=1}^4 
  [d'_i+1]_q,
    \]
where $\wt$ is the weight function we defined in Definition
\ref{def:matching_features} (iii). We have $\wt(M)=h(M)$ by Proposition
\ref{prop:wt} and $h(M)=\rho(R(M))$ for any $M\in \calm$ by Proposition
\ref{prop:mob_level}, where $h$ is the height function from Definition
\ref{def:matching_features} (ii), so it follows that \[ \sum_{R\in
\sigma\inverse(\gamma)}q^{\rho(R)} = q^{d'_1+d'_2+d'_3+d'_4}\prod_{i=1}^4
[d'_i+1]_q. \] The fact that $\mset{d_i:1\le i\le 4}=\mset{d'_i: 1\le i\le 4}$
by Lemma \ref{lem:d_d'} (ii) then implies that \[ \sum_{R\in
  \sigma\inverse(\gamma)}q^{\rho(R)} = q^{d'_1+d'_2+d'_3+d'_4}\prod_{i=1}^4
  [d'_i+1]_q= q^{d_1+d_2+d_3+d_4}\prod_{i=1}^4
[d_i+1]_q=q^{\abs{\seqd}}\prod_{i=1}^4 [d_i+1]_q, \] which proves Equation
\eqref{eq:seq_to_gen} in this case. 

Now suppose that $\seqd=(d_1,d_2,d_3,d_4)\in \cals_F$, so that $\seqd=(1,2,4,j)$ for some
$0\le j\le 7$, and $\gamma=4e_9+2e_{8-j}$ by Proposition \ref{prop:seq_to_sum}
(ii). By \cite[Proposition 3.7, Proposition 4.12]{the240}, if $\calf'$ is the
set of all $30$ inequivalent labellings of the Fano plane with labels from a
fixed $7$-element set, then we have 
\begin{equation}
\label{eq:F7}
\sum_{F \in \calf'} q^{h(F)} = [2]_q [3]_q [5]_q.
\end{equation}
Define
\[
  \calf_y=\{F\in \calf_\gamma: \lc(F)=\{8-j,y\}\}
\]
for each $8-j<y\le 9$. We have
\[
\rho(R(F))=24-(8-j)-y+h(F)=h(F)+(j+16)-y,
\]
for all $F\in \calf_y$
by Proposition \ref{prop:fob_level}, and it then follows from Equation \eqref{eq:F7} that 
\begin{eqnarray*}
  \sum_{R\in \sigma\inverse(\gamma)}q^{\rho(R)}&=& \sum_{y:8-j<y\le 9}\sum_{F\in
  \calf_y}q^{h(F)+(j+16)-y}\\  
                                               &=& [2]_q [3]_q [5]_q\cdot \left(\sum_{y:-9\le -y< {j-8}}
  q^{(j+16)-y}\right)\\
                  &=& [2]_q [3]_q [5]_q\cdot  \left(\sum_{z=j+7}^{2j+7} q^{z}\right) \\ &= &q^{j+7}
[2]_q [3]_q [5]_q [j+1]_q\\ & = & q^{\abs{\seqd}}\prod_{i=1}^4 [d_i+1]_q.
\end{eqnarray*} 
This completes the proof. \end{proof}

\begin{theorem}
\label{thm:endgame}
We have
\begin{equation}
\label{eq:mob_gen}
\sum_{R\in \mob}q^{\rho(R)}=[3]_q[5]_q[7]_q[9]_q
\end{equation}
and 
\begin{equation}
\label{eq:fob_gen}
\sum_{R\in \fob}q^{\rho(R)}=q^7[3]_q[5]_q[8]_q[9]_q.
\end{equation}
The Poincar\'e series of $\ob$ is given by 
\[
PS_\ob(q)=[3]_q[5]_q[9]_q[15]_q.
\]
\end{theorem}

\begin{proof}
  As explained in the proof of Theorem \ref{thm:gen_func}, we have  
  \[
    \sum_{R\in \mob}q^{\rho(R)} = \sum_{M\in \calm}q^{\wt(M)}, 
  \]
  where the generating function on the right hand side is also studied in
  \cite{watson14}. It follows from \cite[Theorem 1]{watson14} that $\sum_{M\in
  \calm}q^{\wt(M)}=[3]_q[5]_q[7]_q[9]_q$, and Equation \eqref{eq:mob_gen}
  follows. 
 
  We have $\fob=\cup_{\seqd\in \cals_F} \sigma\inverse(\Psi(\seqd))$ where
  $\cals_F=\{(1,2,4,j): 0\le j \le 7\}$  by Lemma
  \ref{lemm:fob_sum}, Corollary \ref{cor:all}, and Proposition
  \ref{prop:seq_to_sum} (ii). A routine calculation shows that 
  \[
  \sum_{k=1}^8 q^{k-1}[k]_q = \frac{[8]_q[9]_q}{[2]_q}, 
  \]
and it then follows that
\[
\sum_{R\in \fob}q^{\rho(R)} = \sum_{j=0}^7
q^{j+7}[2][3][5][j+1]_q=q^7[3]_q[5]_q[8]_q[9]_q.
\]
This proves the first sentence of the theorem.

  It follows from the definition of quantum integers that
  $[7]_q+q^7[8]_q=[15]_q$. Since $\ob$ is the disjoint union of $\mob$ and
  $\fob$, Equations \eqref{eq:mob_gen} and \eqref{eq:fob_gen} imply that
  \[
    PS_\ob(q)=\sum_{R\in\ob}q^{\rho(R)}
    =[3]_q[5]_q[7]_q[9]_q
    + q^7[3]_q[5]_q[8]_q[9]_q=[3]_q[5]_q[9]_q[15]_q,
  \]
  which completes the proof.
\end{proof}

\section{Concluding remarks}
\label{sec:conclude}

\begin{rmk}\label{rmk:e7}
The results of this paper have direct analogues for the root system
$\Phi(E_7)$ of type $E_7$ and the collection $\ob'$ of the positive orthogonal
bases of type $E_7$, i.e., the collection of size-7 subsets of $\Phi(E_7)$ that
consist of
mutually orthogonal positive roots. The set $\ob'$ is a quasiparabolic
$W(E_7)$-set of size $135$ by the results of \cite{gx5,gx6}, and we can mimic
the methods of this paper to decompose $\ob'$ as the disjoint union of two
subsets $\mob'$ and $\fob'$, where $\mob'$ contains 105 orthogonal bases modelled by the perfect
matchings of the set $\{0,1,2,3,4,5,6,7\}$ and $\fob'$ contains 30
orthogonal bases modelled by the 30 inequivalent labellings of the Fano plane using
labels from the set $\{1,2,3,4,5,6,7\}$. The map $\sigma$ has $15$ fibres, 
with $\fob'$ forming a single fibre and 
$\mob'$ being the union of 14 fibres that correspond naturally to the 14 Dyck
paths of semilength 4. More generally, we can uniformly index all the 15 fibres
of $\sigma$ by 15 suitable length-3 sequences
obtained via a chain decomposition of the heap of a fully commutative
element, namely, the so-called ``nonnesting element" of type $E_7$ from
\cite[Section 6.2]{gx5}. This indexing scheme and the 
decomposition $\ob'=\mob'\cup \fob'$ are new, and they can be used to prove that
$PS_{\ob'}(q)=[3]_q [5]_q [9]_q$.
\end{rmk}

\begin{rmk}\label{rmk:restriction}
As mentioned in the introduction, there are many interesting examples of
quasiparabolic sets that can be constructed from the orthogonal sets of roots of
type $E_n$, including certain configurations of curves in del Pezzo surfaces of
degrees 2 and 3. The poset $(\ob, \leq_Q)$ is the largest of these exceptional
configurations, and the results of this paper can be used to gain insight into
these smaller examples by restriction. For example, instead of mimicking the
methods of this paper to prove the facts about the set $\ob'$ in Remark
\ref{rmk:e7} from scratch, one may also deduce them from the results of this
paper by identifying $\Phi(E_7)$ as the set of roots $\al=\sum_{i=1}^8
c_i\al_i\in \Phi(E_8)$ where $c_8=0$, or as the set of roots in $\Phi(E_8)$ that
are orthogonal to the highest root $\theta$ of $\Phi(E_8)$. For another example,
one can use the matching and Fano models for $\ob$ to study the set $T(E_8,8)$
from \cite[Section 5.1]{gx6}. This is a set of 630 quadruples of pairwise
orthogonal roots that are closely related to the aforementioned curve
configurations and that form another quasiparabolic set for $W(E_8)$. We can use
the matching and Fano models to obtain the level generating function for $T(E_8,
8)$ established in \cite[Remark 5.10 (i)]{gx6}. \end{rmk}
%
%

\subsection*{Acknowledgements}
We used the SageMath software system \cite{sage} for experimentation and
computer verification of theorems \ref{thm:gen_func}, \ref{thm:mob_res},
\ref{thm:fob_res}, and \ref{thm:endgame}. The second-named author was partially
supported by an AMS--Simons Research Enhancement Grant for PUI Faculty.

\subsection*{AI disclosure}
Claude Opus 5.5 and GPT-6 Astra were used for proofreading drafts of this paper.
We corrected minor errors and made local expository changes based on their
input, but we did not use AI tools in any other form. In particular, we did not
use AI while writing the initial draft of the paper, and we wrote the SageMath
code used during the preparation of the paper without AI assistance.

\appendix
\begin{landscape}
\section{Wide tables}
\label{sec:wide_tables}

\begin{table}[b]
      \centering
\begin{tabular}{@{}lll@{}}
\toprule
  $\al$ & $\be\in \pra$ & $\be(\al)$ \\
  \midrule
  $p_{ij}, j=y$ & $m_{y9}$  & $p_{i9}<0$ \\
  $p_{ij}, j=z$ & $m_{i'j'z}$, where $(i',j')$ is any $i$-excluding middle arc& $p_{ii'j'}<0$ \\
  $p_{ij}, i\in \{y,z\}, j\notin\{y,z\}$ & $p_{yz9}$ & $m_{xj9}<0$, where $\{x,i\}=\{y,z\}$ \\
  $p_{ij}, \{i,j\}\cap\{y,z\}=\emptyset$ & $m_{j\tau(j)z}$ & $p_{i\tau(j)z}<0$\\
  \midrule
  $m_{i9}, i=z$ & $m_{y9}$, or $m_{i_aj_az}$ where $a\in \{1,2,3\}$ &
  $m_{iy}=p_{yz}$, or $p_{i_aj_a9}>0$ \\
  $m_{i9}, i\neq z$ & $m_{y9}, p_{yz9}, r_{i\tau(i)}, m_{i\tau(i)z}$ &
$m_{iy}=p_{yi}, m_{iyz}>0, m_{\tau(i)9}>0, p_{\tau(i)z9}>0$ \\
  \midrule
  $p_{ij9}, y\notin \{i,j\}$ & $m_{y9}$ & $p_{ijy}<0$\\
  $p_{iy9}, i\notin \{y,z\}$ & $r_{i\tau(i)}, p_{yz9}$, or $m_{i_aj_az}$ for the
   $i$-excluding middle arcs $\{i_a,j_a\}$  & $p_{\tau(i)y9}>0, p_{iz}$, or $m_{i_aj_az}(p_{iy9})>0$\\
  \midrule
  $m_{ijk}, y\in \{i,j,k\}$ & $m_{y9}$ & $m_{ab9}<0$, where
  $\{a,b,y\}=\{i,j,k\}$\\
  $m_{ijk}, \{y,z\}\cap \{i,j,k\}=\emptyset$ & $p_{yz9}$ & $p_{yz9}(m_{ijk})<0$\\
  $m_{ijz}, y\notin \{i,j,z\}$ & $r_{i\tau(i)}, r_{j\tau(j)}, m_{i\tau(i)z},
  m_{j\tau(j)z}$ & $m_{\tau(i)jz}>0, m_{i\tau(j)z}>0, p_{\tau(i)j},
  p_{\tau(j)i}$\\
  \bottomrule
\end{tabular}
\caption{Analysis of $\be(\al)$, II}
\label{tab:mob2}
\end{table}
\begin{table}[b]
  \centering
\begin{tabular}{@{}lll@{}}
\toprule
  $\al$ & $\be\in \pra$ & $\be(\al)$ \\
  \midrule
  $p_{ij}, i\in L(F), \tau(i,9)=j$ & $m_{i'jk'}$, where $\{i',j,k'\}$ is any
 $i$-excluding block & $p_{ii'k'}<0$ \\
  $p_{ij}, i\in L(F),\tau(i,9)\neq j$ & $p_{ij'9}$, where $j'=\tau(i,9)$ & $m_{jj'9}<0$\\
  $p_{ij}, i\notin L(F), j\in L(F)$ & $m_{i'jk'}$, where $\{i',j,k'\}$ is any
 $9$-excluding block& $p_{ii'k'}<0$\\
  \midrule
  $m_{i9}, i\in \{x,y\}$ & $p_{xy}$, or $p_{i'j'9}$ for the blocks of the form
  $\{i',j',9\}$ & $m_{z9}>0$ where $\{i,z\}=\{x,y\}$, or $m_{ii'j'}>0$\\
  $m_{i9}, i\notin\{x,y\}$ & $r_{i'j'k'}$ for the blocks $\{i',j',k'\}$ such that $\abs{\{i',j',k'\}\cap \{i,9\}}=1$& $r_{i'j'k'}(m_{i9})>0$ \\
  \midrule
  $p_{xy9}$ & $m_{i'j'k'}$ for the $9$-excluding blocks $\{i',j',k'\}$  & $m_{i'j'k'}(p_{xy9})>0$\\
  $p_{iz9}$, $i\notin \{x,y\},z\in \{x,y\}$ & $p_{xy}$, $p_{i\tau(i,9)9}$, or
  $m_{i'j'\tau(i,9)}$ for the $i$-excluding blocks $\{i',j',\tau(i,9)\}$& $p_{iz'9}>0$,
$p_{z\tau(i,9)}$, or $m_{i'j'\tau(i,9)}(p_{iz9})>0$\\ 
  $p_{ij9}, \{i,j\}\cap \{x,y\}=\emptyset$ &
  $\gamma:=m_{\tau(i,9)\tau(j,9)\tau(i,j)}$,$m_{ij\tau(i,j)}$,
  $p_{i\tau(i,9)9}$, $p_{j\tau(j,9)9}$ & $\gamma(p_{ij9})>0$, $m_{\tau(i,j)9}>0$,
  $p_{j\tau(i,9)}, p_{i\tau(j,9)}$\\
  \midrule
  $m_{ijk}$, Case (1) & $p_{i'j'9}$, where $\{i',j',9\}\cap\{i,j,k\}=\emptyset$
  (see the proof of Theorem \ref{thm:fob_res}) & $p_{i'j'9}(m_{ijk})<0$\\
  $m_{ijk}$, Case (2) & $\gamma:=m_{\tau(i,j)\tau(i,k)\tau(j,k)}, m_{ij\tau(i,j)},m_{ik\tau(i,k)}, m_{jk\tau(j,k)}
  $ & $\gamma(m_{ijk})>0$, $p_{\tau(i,j)k},p_{\tau(i,k)j}, p_{\tau(j,k)i}$\\
  \bottomrule
\end{tabular}
\caption{Analysis of $\be(\al)$, IV}
\label{tab:fob2}
\end{table}
\end{landscape}
\bibliographystyle{plain} \bibliography{gx7.bib}

\end{document}